\documentclass[12pt]{article}

\usepackage{amsmath,amssymb,amscd,amsthm,makeidx,txfonts,graphicx,url,bm}
\usepackage{xypic}
\usepackage[backref]{hyperref}

\numberwithin{equation}{section}

\let\oldmarginpar\marginpar
\renewcommand\marginpar[1]{\-\oldmarginpar[\raggedleft\footnotesize #1]
{\raggedright\footnotesize #1}}

\makeindex

\xyoption{all}
\input{xypic}

\newtheorem{theorem}{Theorem}[section]
\newtheorem{proposition}[theorem]{Proposition}
\newtheorem{corollary}[theorem]{Corollary}
\newtheorem{conjecture}[theorem]{Conjecture}
\newtheorem{lemma}[theorem]{Lemma}

\theoremstyle{remark}
\newtheorem{remark}[theorem]{Remark}

\theoremstyle{definition}

\newcounter{margin}
{\end{itshape}  \bigskip}

\def\beq{\begin{eqnarray}}
\def\eeq{\end{eqnarray}}
\def\bes{\begin{eqnarray*}}
\def\ees{\end{eqnarray*}}
\def\vv{{\bf v}}

\def\muhat{{\bm \mu}}

\def\lambdahat{{\bm \lambda}}
\def\betahat{{\bm \beta}}
\def\alphahat{{\bm \alpha}}

\def\omhat{{\bm \omega}}

\def\Apol{Kac polynomial }
\def\Apols{Kac polynomials }

\DeclareMathOperator{\gen}{gen}
\DeclareMathOperator{\reg}{reg}
\DeclareMathOperator{\Trace}{Tr}
\DeclareMathOperator{\Rep}{Rep}
\DeclareMathOperator{\opp}{opp}
\DeclareMathOperator{\ch}{ch}

\def\C{\mathbb{C}}
\def\calK{\mathcal{K}}
\def\M{{\mathcal{M}}}
\def\calQ{{\mathcal{Q}}}
\def\calA{{\mathcal{A}}}

\def\calF{{\mathcal{F}}}

\def\pihat{{\bf \pi}}
\def\calP{\mathcal{P}}
\def\calH{\mathcal{H}}
\def\calZ{\mathcal{Z}}
\def\bG{\mathbb{G}}
\def\DT{{\rm DT}}

\def\z{{\mathbf{x}}}

\def\x{\mathbf{x}}

\def\v{\mathbf{v}}

\def\P{\mathcal{P}}

\def\bH{\mathbb{H}}
\def\calE{\mathcal{E}}
\def\Hs{\mathbb{H}^s}

\def\N{\mathbb{Z}_{\geq 0}}

\def\F{\mathbb{F}}
\def\Q{\mathbb{Q}}

\def\T{\mathbb{T}}

\def\t{\mathfrak{t}}
\def\tGamma{\tilde{\Gamma}}
\def\tI{\tilde{I}}
\def\tv{{\tilde{\v}}}
\def\talpha{{\cup\muhat}}
\def\tomega{{\tilde{\omega}}}
\def\tomega{\omega}

\def\calC{{\mathcal C}}
\def\calO{{\mathcal O}}

\def\Z{\mathbb{Z}}

\def\K{\mathbb{K}}

\def\gl{{\mathfrak g\mathfrak l}}

\newcommand{\nc}{\newcommand}

\nc{\op}[1]{\mathop{\mathchoice{\mbox{\rm #1}}{\mbox{\rm #1}}
{\mbox{\rm \scriptsize #1}}{\mbox{\rm \tiny #1}}}\nolimits}
\nc{\al}{\alpha}

\nc{\ep}{\varepsilon} \nc{\ga}{\gamma} \nc{\Ga}{\Gamma}
\nc{\la}{\lambda} \nc{\La}{\Lambda} \nc{\si}{\sigma}
\nc{\Sig}{{\Gamma}} \nc{\Om}{\Omega} \nc{\om}{\omega}

\nc{\SL}{{\rm SL}} \nc{\GL}{{\rm GL}} \nc{\PGL}{{\rm PGL}}

\nc{\cpt}{{\op{cpt}}} \nc{\Dol}{{\op{Dol}}} \nc{\DR}{{\op{DR}}}
\nc{\B}{{\op{B}}} \nc{\Triv}{\op{Triv}} \nc{\Hod}{{\op{Hod}}}
\nc{\Log}{{\op{Log}}} \nc{\Exp}{{\op{Exp}}} \nc{\Est}{E_{\op{st}}}
\nc{\Hst}{H_{\op{st}}} \nc{\Left}[1]{\hbox{$\left#1\vbox to
  10.5pt{}\right.\nulldelimiterspace=0pt \mathsurround=0pt$}}
\nc{\Right}[1]{\hbox{$\left.\vbox to
  10.5pt{}\right#1\nulldelimiterspace=0pt \mathsurround=0pt$}}
\nc{\LEFT}[1]{\hbox{$\left#1\vbox to
  15.5pt{}\right.\nulldelimiterspace=0pt \mathsurround=0pt$}}
\nc{\RIGHT}[1]{\hbox{$\left.\vbox to
  15.5pt{}\right#1\nulldelimiterspace=0pt \mathsurround=0pt$}}

\nc{\bee}{{\bf E}} \nc{\bphi}{{\bf \Phi}}

\begin{document}

\title{Positivity of \Apols and DT-invariants for quivers}

\author{ Tam\'as Hausel
\\ {\it EPF Lausanne}
\\{\tt tamas.hausel@epfl.ch} \and Emmanuel Letellier \\ {\it
  Universit\'e de Caen} \\ {\tt
  letellier.emmanuel@math.unicaen.fr}\and Fernando Rodriguez-Villegas
\\ 
{\it University of Texas at Austin} \\ {\tt
  villegas@math.utexas.edu}\\ \\
 }

\pagestyle{myheadings}

\maketitle

\begin{abstract}

  We give a cohomological interpretation of both the \Apol
  and the refined Donaldson-Thomas-invariants of quivers. This
  interpretation yields a proof of a conjecture of Kac from 1982 and
  gives a new perspective on recent work of Kontsevich--Soibelman.
  This is achieved by computing, via an arithmetic Fourier transform,
  the dimensions of the isoytpical components of the cohomology of
  associated Nakajima quiver varieties under the action of a Weyl
  group.  The generating function of the corresponding Poincar\'e
  polynomials is an extension of Hua's formula for \Apols of
  quivers involving Hall-Littlewood symmetric functions.  The
  resulting formulae contain a wide range of information on the
  geometry of the quiver varieties.
\end{abstract}

\section{The main results}\label{results}

Let $\Gamma=(I,\Omega)$ be a quiver: that is, an oriented graph on a
finite set $I=\{1,\dots,r\}$ with $\Omega$  a finite multiset of oriented edges.
In his study of the representation theory of quivers, Kac \cite{kac2}
introduced $A_\vv(q)$, the number of isomorphism classes of absolutely
indecomposable representations of $\Gamma$ over the finite field
$\F_q$ of dimension $\vv=(v_1,\ldots,v_r)$ and showed they are
polynomials in $q$. We call $A_\vv(q)$ the {\it Kac polynomial} for
$\Gamma$ and $\vv$.  Following ideas of Kac \cite{kac2}, Hua
\cite{hua} proved the following generating function identity:
\beq \label{hua} \sum_{\vv\in \N^r\backslash\{0\}} A_\vv(q)
T^\vv=(q-1)\Log\left( \sum_{\pihat=(\pi^1,\dots,\pi^r) \in \calP^r}
  \frac{\prod_{i\rightarrow j\in \Omega} q^{\langle
      \pi^i,\pi^j\rangle}} {\prod_{i\in I}
    q^{\langle\pi^i,\pi^i\rangle}
    \prod_{k}\prod_{j=1}^{m_k(\pi^i)}(1-q^{-j})
  }T^{|\pihat|}\right),\eeq where $\calP$ denotes the set of
partitions of all positive integers, $\Log$ is the plethystic
logarithm (see \cite[\S 2.3.3]{aha}), $\langle,\rangle$ is the pairing
on partitions defined by
$$
\langle \lambda,\mu\rangle=\sum_{i,j}
\min(i,j) m_i(\lambda)m_j(\mu)
$$
with $m_j(\lambda)$ the multiplicity of the part $j$ in the partition
$\lambda$, $T^\vv:=T_1^{v_1}\cdots T_r^{v_r}$ for some variables $T_i$
and finally $|\pihat|:=(|\pi_1|\cdots |\pi_r|)$

Using such generating functions Kac \cite{kac2} proved that in fact
$A_\vv(q)$ has integer coefficients and formulated two main
conjectures.  First, he conjectured that for quivers with no loops the
constant term $A_\vv(0)$ equals the multiplicity of the root $\vv$ in
the corresponding Kac-Moody algebra. The proof of this conjecture was
completed in \cite{hausel-kac}.  We will give a general proof of Kac's
second conjecture here:

\begin{conjecture}[{\rm \cite[Conjecture 2]{kac2}}] \label{kac} The
  \Apol $A_\vv(q)$ has non-negative coefficients.  
\end{conjecture}

This conjecture was settled for indivisible dimension vectors and any
quiver by Crawley-Boevey and Van den Bergh \cite{crawley-etal} in
2004; they gave a cohomological interpretation of the \Apol
for indivisible dimension vectors in terms of the cohomology of an
associated Nakajima quiver variety. More precisely
\cite[(1.1)]{crawley-etal}, they showed that for $\vv$ indivisible
\beq \label{indiv} A_\vv(q)=\sum \dim
\left(H_c^{2i}(\calQ_\vv;\C)\right) q^{i-d_\vv},\eeq where $\calQ_\vv$
is a certain smooth {\em generic} complex quiver variety of dimension
$2d_\vv$. Similarly, our proof of the general case will follow by interpreting
the coefficients of $A_\vv(q)$ as the dimensions of the sign
isotypical component of cohomology groups of a smooth generic quiver
variety $\calQ_{\tilde \vv}$ attached to an extended quiver
(see~\eqref{A-pol}).

Recently, Mozgovoy \cite{mozgovoy} proved Conjecture~\ref{kac} for
any dimension vector for quivers with at least one loop at each
vertex.  His approach uses Efimov's proof \cite{efimov} of a
conjecture of Kontsevich--Soibelman \cite{kontsevich-soibelman2} which
implies positivity for certain refined DT-invariants of symmetric
quivers with no potential.

The goal of Kontsevich--Soibelman's theory is to attach refined (or
motivic, or quantum) Donaldson--Thomas invariants (or DT-invariants
for short) to Calabi--Yau 3-folds $X$. The invariants should only
depend on the derived category of coherent sheaves on $X$ and some
extra data; this raises the possibility of defining DT-invariants for
certain Calabi-Yau $3$-categories which share the formal properties of
the geometric situation, but are algebraically easier to study. The
simplest of such examples are the Calabi-Yau $3$-categories attached
to quivers (symmetric or not) with no potential (c.f. \cite{ginzburg}).

Denote by $\overline{\Gamma}=(I,\overline{\Omega})$ the double quiver,
that is $\overline{\Omega}=\Omega\coprod \Omega^{\opp}$, where
$\Omega^{\opp}$ is obtained by reversing all edges in $\Omega$.  The
refined DT-invariants  of $\overline{\Gamma}$ (a slight
renormalization of those introduced by Kontsevich and Soibelman
\cite{kontsevich-soibelman1}) are defined by the following
combinatorial construction. For $\vv=(v_1,\ldots,v_r)\in \N^r$ let
$$
\delta(\vv):=\sum_{i=1}^rv_i, \qquad
\gamma(\vv):=\sum_{i=1}^rv_i^2-
\sum_{i\rightarrow j\in \Omega}v_iv_j .
$$
Then 
\begin{equation}
\label{DT-inv-defn}
\sum_{\v\in \N^r\backslash\{0\}}
\DT_\v(q)\,(-1)^{\delta(\v)} T^\v 
:=  (q-1)\Log \sum_{\v\in
  \N^r}\frac{(-1)^{\delta(\v)}  
  q^{-\frac 12(\gamma(\v)+\delta(\v))}}{\prod_{i=1}^r (1-q^{-1})\cdots
  (1-q^{-v_i}) }\,T^\v
\end{equation}
It was proved
in \cite{kontsevich-soibelman2} that $\DT_\vv(q)\in \Z[q,q^{-1}]$. In
fact, as a consequence of Efimov's proof \cite{efimov} of
\cite[Conjecture 1]{kontsevich-soibelman2}, $\DT_\vv(q)$ actually has
non-negative coefficients. We will give an alternative proof of this
in~\eqref{even} by interpreting its coefficients as dimensions of
cohomology groups of an associated quiver variety.

 \begin{remark}
\label{parity}
  We should stress that we have restricted to double quivers for the
  benefit of exposition; our results extend easily to any symmetric
  quiver. We outline how to treat the general case in~\S\ref{DT2}.
 \end{remark}

The technical starting point in this paper is a common generalization
of \eqref{DT-inv-defn} and Hua's formula \eqref{hua}. Namely we
consider 
\beq \label{maingen}\bH(\x_1,\dots,\x_r;q):=(q-1)\Log\left(
  \sum_{\pi=(\pi^1,\dots,\pi^r) \in \calP^r} \frac{\prod_{i\rightarrow
      j\in \Omega} q^{\langle \pi^i,\pi^j\rangle}} {\prod_{i\in I}
    q^{\langle\pi^i,\pi^i\rangle}
    \prod_{k}\prod_{j=1}^{m_k(\pi^i)}(1-q^{-j}) } \prod_{i=1}^r
  \tilde{H}_{\pi^i}(\x_i;q)\right), \eeq 
where $\x_i=(x_{i,1},x_{i,2},\dots)$ is a set of infinitely many independent
variables, and $\tilde{H}_{\pi^i}(\x_i;q)$ denote the (transformed)
Hall-Littlewood polynomial (see \cite[\S 2.3.2]{aha}), which is a
symmetric function in $\x_i$ and polynomial in~$q$.  From
\eqref{maingen} we can extract many rational functions in $q$ just by
pairing against other symmetric functions.  For a multi-partition
$\muhat=(\mu^1,\dots,\mu^r)\in \calP^r$ we let
$s_\muhat:=s_{\mu^1}(\x_1)\cdots s_{\mu^r}(\x_{r})$, where
$s_{\mu^i}(\x_i)$ is the Schur symmetric function attached to the
partition $\mu^i$. Define \beq \label{hum}\Hs_\muhat(q):=\left\langle
  \bH(\x_1,\dots,\x_r;q), s_\muhat\right\rangle, \eeq where
$\langle,\rangle$ is the natural extension to $r$ variables of the
Hall pairing on symmetric functions, defined by declaring the basis
$s_\muhat$ orthonormal (see \cite[(2.3.1)]{aha}). Note that a priori
the $\Hs_\muhat$ are rational functions in $q$; we will prove that
they are actually polynomials with non-negative integer coefficients.
More precisely, we will show (Theorem~\ref{main} (i)) that the
coefficients of $\Hs_\muhat(q)$ are the dimensions of certain
cohomology groups.

Our first formal observation is the
following
\begin{proposition}\label{formal} 
  For any $\vv=(v_1,\ldots,v_r)\in \N^r$ we have

(i)
 \beq
  A_\vv(q)=\Hs_{\vv^1}(q),\eeq where $\vv^1:=((v_1),\dots,(v_r))\in
  \calP^r$ 
 and 

(ii)
\beq \DT_\vv(q)=\Hs_{1^\vv}(q),\label{1.7}\eeq where
  $1^\vv:=((1^{v_1}),\dots,(1^{v_r}))\in \calP^r$.
\end{proposition}
We will prove
  Part (i) as a consequence of Theorem \ref{theoaha2} since
  $h_{\v^1}=s_{\v^1}$ and part (ii) is a special case of
  Proposition~\ref{propDT}.

Fix a non-zero multi-partition $\muhat\in \P^r$ and let
$\vv=|\muhat|:=(|\mu^1|,\ldots,|\mu^r|)$. Associated to the pair
$(\Gamma,\vv)$ we construct a new quiver by attaching a leg of length
$v_i-1$ at the vertex $i$ of $\Gamma$ where $v_i:=|\mu^i|$.  We denote it by
$\tilde{\Gamma}_\vv=(\tilde{I}_\vv, \tilde{\Omega}_\vv)$.  We extend
the dimension vector $\vv: I\to \N$ to $\tilde{\vv}:\tilde{I}_\vv\to
\N$ by placing decreasing dimensions $v_i-1,v_i-2,\dots,1$ at the
extra leg starting with $v_i$ at the original vertex $i$. We also
consider the subgroup $W_\vv<W$ of the Weyl group of the quiver
generated by the reflections at the extra vertices $\tilde{I}\setminus
I$. We may identify $W_\vv$ with $S_{v_1}\times\dots \times S_{v_r}$,
the Weyl group of the group $\GL_\vv:=\GL_{v_1}\times\dots\times
\GL_{v_r}$.

Because by construction $\tilde{\vv}$ is indivisible we can define the
corresponding smooth {\em generic} complex quiver variety
$\calQ_{\tilde{\vv}}$. Note that $\tilde{\vv}$ is left invariant by
$W_\vv$ and thus $W_\vv$ acts on $H_c^*(\calQ_{\tilde{\vv}};\C)$ by
work of Nakajima \cite{nakajima,nakajima2}, Lusztig \cite{Lusztig}, Maffei
\cite{maffei} and Crawley-Boevey--Holland \cite{crawley-holland}. We denote by $\chi^\muhat=\chi^{\mu^1}\cdots
\chi^{\mu^r}:W_\v\to \C^\times$ the product of the irreducible
characters $\chi^{\mu_i}$ of the symmetric groups $S_{v_i}$ in the
notation of \cite[\S I.7]{macdonald}. In particular, $\chi^{(v_i)}$ is
the trivial character and $\chi^{(1^{v_i})}$ is the sign character
$\epsilon_i:S_{v_i}\to \{\pm 1\}<\C^\times$.  If
$\muhat^\prime:=((\mu^1)^\prime,\dots,(\mu^r)^\prime)$, where
$(\mu^i)^\prime$ is the dual partition of $\mu^i$, then
$\chi^{\muhat^\prime}=\epsilon\chi^\muhat$ with
$\epsilon:=\epsilon_1\cdots\epsilon_r$ the sign character of
$W_\v$. 

We may decompose the representation of $W_\vv$ on
$H_c^*(\calQ_{\tilde{\vv}};\C)$ into its isotypical components
$$
H^*_c(\calQ_{\tilde{\vv}};\C)\cong
\bigoplus_{\muhat \in
  \calP_\vv}H^*_c(\calQ_{\tilde{\vv}};\C)_{\chi^\muhat},
$$
where $\calP_\vv$ denotes the set of multi-partitions
$\muhat=(\mu^1,\dots,\mu^r)$ of size $\vv=(v_1,\dots,v_r)$.

More generally, for a multi-partition
$\muhat=(\mu^1,\dots,\mu^r)\in\calP_\v$, with
$\mu^i=(\mu^i_1,\mu^i_2,\dots,\mu^i_{l_i})$ and $l_i$ the length of
$\mu^i$, denote by $\Gamma_\muhat$ the quiver obtained from
$(\Gamma,\muhat)$ by adding at each vertex $i\in I$ a leg with $l_i-1$
edges. We denote by $\v_\muhat$ the dimension vector of
$\Gamma_\muhat$ with coordinates
$v_i,v_i-\mu^i_1,v_i-\mu^i_1-\mu^i_2,\dots, \mu^i_{l_i}$ at the $i$-th
leg. 
Define
\begin{equation}
\label{dim-defn}
d_\muhat:=1-\tfrac12 {^t}\v_\muhat {\bf C_\muhat}\v_\muhat
\end{equation}
with ${\bf C}_\muhat$ the Cartan matrix of $\Gamma_\muhat$.
Notice that if $\muhat=(1^\v)$ then 
$\tilde{\Gamma}_\v=\Gamma_{1^\v}, \v_\muhat=\tilde \v$; 
 we will write $d_{\tilde \v}$ instead of $d_{1^\v}$.
The quiver variety $\calQ_{\tilde \v}$ is non-empty if and only if
$\tilde \v$ is a root of $\Gamma_{\tilde \v}$ in which case  it has
dimension $2d_{\tilde \v}$ \cite[Theorem 1.2]{CB}.

Our main geometric result is the following 
\begin{theorem} 
\label{main}
We have

(i)
$$
\Hs_\muhat(q)=\sum_i
  \dim\left(H^{2i}_c(\calQ_{\tilde{\vv}};\C)_{\epsilon\chi^\muhat}\right)
  q^{i-d_{\tilde{\vv}}}.
$$

(ii) $\Hs_\muhat(q)$ is non-zero if and only if $\v_\muhat$ is a root
of $\Gamma_\muhat$, in which case it is a monic polynomial of
degree~$d_\muhat$. Moreover, $\Hs_\muhat(q)=1$ if and only if
$\v_\muhat$ is a real root.
\label{maintheo}\end{theorem}

\begin{remark}
1.  Another way to phrase Theorem~\ref{main} (i) is as follows. Let
  $V_{\v, i}$ be the representation $H_c^i(\calQ_\tv;\C)$ of $W_\v$
  tensored with the sign representation and consider the graded
  representation $V:=\bigoplus_{\v,i} V_{\v,i}$. Then the
  Frobenius-Hilbert series of $V$ is $\bH$. In other words,
$$
\sum_{\v,i}\ch(V_{\v,i})q^{i-d_{\tilde v}}= \bH(\x_1,\dots,\x_r;q),
$$
where $\ch$ is the Frobenius characteristic map (naturally extended to
the product of symmetric groups $W_\v$).  

2. Theorem~\ref{main} (ii) provides a criterion (in terms of root
systems) for the appearance of an irreducible character of $W_\v$ in
$\bigoplus_i V_{\v,i}$.
\end{remark}

In combination with Proposition~\ref{formal} Theorem~\ref{main}
implies the following.
\begin{corollary} 
\label{pols-corollary}
We have

(i)
\begin{equation}
\label{A-pol}
A_\vv(q)=\sum_i
  \dim\left(H^{2i}_c(\calQ_{\tilde{\vv}};\C)_{\epsilon}\right)
  q^{i-d_{\tilde{\vv}}},
\end{equation}
 and
\begin{equation}
\label{even}
\DT_\vv(q)=\sum_i
  \dim\left(H^{2i}_c(\calQ_{\tilde{\vv}};\C)^{W_\vv}\right)
  q^{i-d_{\tilde{\vv}}},
\end{equation}

(ii) In particular, $A_\vv(q)$ and $\DT_\vv(q)$ have non-negative
integer coefficients.

(iii) Conjecture~\ref{kac} holds for any quiver and dimension vector
$\vv$.

(iv) The polynomial $\DT_\v(q)$ is non-zero if and only if $\tv$ is a
root of $\tilde{\Gamma}_\v$, in which case it is monic of
degree~$d_\tv$. Moreover, $\DT_\v(q)=1$ if and only if $\tv$ is a real
root.
\end{corollary}

\begin{remark} 
  1. In the case when the dimension vector $\v$ is indivisible, we
  could have proved Theorem~\ref{main} (i) using the ideas of
  \cite{letellier2} based on the theory of perverse sheaves. For more
  details see Remark~\ref{remark}.

  2. As mentioned above, the fact that $\DT_\vv(q)$ has non-negative
  coefficients also follows from \cite[Conjecture
  1]{kontsevich-soibelman2} proved by Efimov~\cite{efimov}.

3. Mozgovoy in \cite{mozgovoy2} shows that $A_\vv(q)$ can also be
interpreted as the refined DT-invariant of an associated quiver
$\hat{\Gamma}$ with a non-trivial potential. It raises the question of
what is the meaning of our more general polynomials $\Hs_\muhat(q)$ in
refined DT-theory.

4. For a star-shaped quiver the series~\eqref{maingen} is shown
in~\cite{aha} to be the pure part (the specialization $(q,t)\mapsto
(0,q)$) of a series in two variables $q,t$. We expect that there is an
analogous $t$-deformation of~\eqref{maingen} for any quiver. Under
this deformation the Hall-Littlewood polynomials would be replaced by
the Macdonald polynomials; the challenge is to $t$-deform the rest of
the terms. Conjecturally, the geometrical meaning of such a $(q,t)$
formula should involve the mixed Hodge polynomials of multiplicative
quiver varieties of Crawley-Boevey and
Shaw~\cite{crawley-boevey-shaw}. We also expect that the natural Weyl
group action on the cohomology of such multiplicative quiver varieties
extends, at least in some special cases, to representations of
rational Cherednik algebras.

5. Finally, we mention that in some cases such $(q,t)$-formulas were found recently
by Chuang-Diaconescu-Pan \cite{diaconescu-etal} to conjecturally describe
refined BPS invariants and, in particular, refined DT-invariants of
local curve Calabi-Yau $3$-fold geometries.

\end{remark}

\paragraph{Acknowledgements.} We would like to thank Sergey Mozgovoy
for explaining his papers \cite{mozgovoy,mozgovoy2} and for useful comments.
We thank 
William Crawley-Boevey, Bernard Keller,  Maxim Kontsevich, Andrea Maffei, Hiraku Nakajima, Philippe Satg\' e, Yan Soibelman and Bal\'azs Szendr{\H o}i 
for useful comments.  The first author
thanks the Royal Society for funding his research 2005-2012 in the
form of a Royal Society University Research Fellowship as well as the
Mathematical Institute and Wadham College in Oxford for a very
productive environment. The second author is supported by Agence
Nationale de la Recherche grant ANR-09-JCJC-0102-01. He would like
also to thank the Mathematical Institute and Wadham College in Oxford
where some of this work was done. The third author is supported by the
NSF grant DMS-1101484 and a Research Scholarship from the Clay
Mathematical Institute.
 
\section{Proof of Theorem \ref{maintheo}}\label{proofmain}

\subsection{The quiver varieties $\calQ_\tv$}\label{supernova}

Let $\Gamma, I, \Omega$ be as in the introduction. Let $\K$ be any
algebraically closed field. For a dimension vector $\v=(v_i)_{i\in
  I}\in\N^I$ put
$$
\K^\v:=\bigoplus_{i\in I}\K^{v_i},\hspace{.5cm}\GL_\v:=\prod_{i\in
  I}\GL_{v_i}(\K),\hspace{.5cm}\gl_\v:=\bigoplus_{i\in I}\gl_{v_i}(\K).
$$
By a graded subspace of $V\subseteq \K^\v$ we will mean a subspace of
the form
$$
V=\bigoplus_{i\in I} V_i, \qquad \qquad V_i\subseteq \K^{v_i}.
$$

The group $\GL_\v$ acts on $\gl_\v$ by conjugation. For an element
$X=(X_i)_{i\in I}\in\gl_\v$ we put $\Trace(X):=\sum_{i\in
  I}\Trace(X_i)$. We denote by $T_\v$ the maximal torus of
$\GL_\v$ whose elements are of the form $(g_i)_{i\in I}$ with $g_i$ a
diagonal matrix for each $i\in I$. The Weyl group
$W_\v:=N_{\GL_\v}(T_\v)/T_\v$ of $\GL_\v$ with respect to $T_\v$ is
isomorphic to $\prod_{i\in I}S_{v_i}$ where $S_v$ denotes the
symmetric group in $v$ letters. Recall that a semisimple element $X\in
\gl_\v$ is \emph{regular} if $C_{\GL_\v}(X)$ is a maximal torus of $\GL_\v$, i.e., the
eigenvalues of the coordinates of $X$ are all with multiplicity $1$.

We say that an adjoint orbit $\calO$ of $\gl_\v$ is \emph{generic}
if $\Trace(X)=0$ and if for any graded 
subspace $V\subseteq \K^\v$
stable by some $X\in \calO$ such that
$$
\Trace\left(X|_V\right)=0
$$
then either $V=0$ or $V=\K^\v$. We fix such a generic regular semisimple
adjoint orbit $\calO\subset\gl_\v$ (we can prove as in  \cite[\S 2.2]{aha} that such a choice is always possible).

Let $\overline{\Gamma}$ be the \emph{double quiver} of $\Gamma$ ;
namely, $\overline{\Gamma}$ has the same vertices as $\Gamma$ but for
each arrow $\gamma\in \Omega$ going from $i$ to $j$ we add a new arrow
$\gamma^*$ going from $j$ to $i$. We denote by
$\overline{\Omega}=\{\gamma,\gamma^*\;|\;\gamma\in \Omega\}$ the set
of arrows of $\overline{\Gamma}$. Consider the space
$$
\Rep_\K\left(\overline{\Gamma},\v\right):=\bigoplus_{i\rightarrow
  j\in \overline{\Omega}}{\rm Mat}_{v_j,v_i}(\K)$$ of representations
of $\overline{\Gamma}$ with dimension $\v$. Recall that $\GL_\v$ acts
on $\Rep_\K\left(\overline{\Gamma},\v\right)$ as

$$
(g\cdot\varphi)_{i\rightarrow j}=g_j\varphi_{i\rightarrow j}g_i^{-1}
$$
for any $g=(g_i)_{i\in I}\in\GL_\v$,
$\varphi=(\varphi_\gamma)_{\gamma\in\overline{\Omega}}\in{\rm
  Rep}_\K\left(\overline{\Gamma},\v\right)$ and any arrow
$i\rightarrow j\in\overline{\Omega}$. 

Let $\tGamma{_\v}$ on vertex set $\tI{_\v}$ be the quiver  obtained from $(\Gamma,\v)$ by adding at each vertex $i\in I$ a leg
 of length $v_i-1$ with the edges all oriented toward the vertex $i$. Define
  $\tv\in \N^{\tI{_\v}}$ as the dimension vector with coordinate $v_i$
  at $i\in I\subset \tI{_\v}$ and with coordinates
  $(v_i-1,v_i-2,\dots,1)$ on the leg attached to the vertex $i\in
  I$.

  Let $\calQ_\tv$ be the quiver variety over $\K$ attached to the quiver $\tGamma$ and
  parameter set  defined from the eigenvalues
  of $\calO$ (see \cite{aha} and the reference therein). 

  Concretely, define the moment map \beq\mu_{\mathbf{v}}:{\rm
    Rep}_{\K}\left(\overline{\Gamma},\mathbf{v}\right)\rightarrow
  \gl_\v^0 \\
  (x_{\gamma})_{\gamma\in\overline{\Omega}}
  \mapsto\sum_{\gamma\in\Omega}[x_{\gamma},x_{\gamma^*}],\label{moment}
  \eeq where
$$
\gl_\v^0:=\left.\left\{X\in\gl_\v\,\right|\,{\rm
    Tr}\,(X)=0\right\}.
$$
Then
$\calQ_\tv$ 
is the affine GIT quotient
\begin{eqnarray}\label{supernovadefined}
\mu_\v^{-1}\left(\calO\right)/\!/\GL_\v:={\rm
  Spec}\,\left(\K[\mu_\v^{-1}(\calO)]^{\GL_\v}\right). 
\end{eqnarray}
Note that the one-dimensional torus $\bG_m$ embeds naturally in
$\GL_\v$ as $t\mapsto (t\cdot I_{v_i})_{i\in I}$ where $I_v$ is the
identity matrix of $\GL_v$. The action of $\GL_\v$ on
$\mu_\v^{-1}\left(\calO\right)$ factorizes through an action of
$G_\v:=\GL_\v/\bG_m$.

We have the following theorem whose proof is similar to that of
\cite[Theorem 2.2.4]{aha}. 

\begin{theorem}
 (i) The variety $\calQ_\tv$ is non-singular and the
  quotient map $\mu^{-1}_\v(\calO)\rightarrow\calQ_\tv$ is a principal
  $G_\v$-bundle in the \'etale topology.

  \noindent (ii) The odd degree   cohomology of $\calQ_\tv$ vanishes.

\label{odd}
\end{theorem}

\subsection{Weyl group action}
\label{Weyl-gp-action}

\subsubsection{Weyl group action}\label{W1}

 Let $\K$ denote an arbitrary algebraically closed field as before. Let $\ell$ be a prime different from the characteristic ${\rm char}\,\K$ of $\K$, and for an algebraic variety over $\K$ denote by $H_c^i(X;\overline{\Q}_\ell)$ the compactly supported $\ell$-adic cohomology.

Denote by $\t_\v^{\gen}$ the generic regular semisimple elements of
the Lie algebra $\t_\v$ of $T_\v$. For $\sigma\in\t_\v^{\gen}$ define
$$
\M_\sigma:=\left\{(\varphi,X,gT_\v)\in{\rm
    Rep}_\K(\Gamma,\v)\times\gl_\v\times(\GL_\v/T_\v)\;|\;
  g^{-1}Xg=\sigma,\, \mu_\v(\varphi)=X\right\}/\!/\GL_\v 
 $$
 where $\GL_\v$ acts by
 $$g\cdot(\varphi,X,hT_\v)=(g\cdot\varphi,gXg^{-1},ghT_\v).$$ The
 following lemma is immediate. 
 
 \begin{lemma} The projection $(\varphi,X,gT_\v)\rightarrow\varphi$
   induces an isomorphism from $\M_\sigma$ onto the  
   quiver
   variety 
   associated to the adjoint orbit of $\sigma$ as in \eqref{supernovadefined}.
\label{Ms}
\end{lemma}

For $w\in W_\v$, denote by $w:\M_\sigma\rightarrow\M_{\dot{w}\sigma\dot{w}^{-1}}$ the isomorphism
$(\varphi,X,gT_\v)\mapsto(\varphi,X,g\dot{w}^{-1}T_\v)$. The first aim of this section is to prove the following theorem.

 \begin{theorem} There exists a prime $p_0$ such that if  ${\rm char}\,\K\geq
   p_0$ or if $\K=\C$, then for  any $\sigma, \tau\in \t_\v^{\gen}$ there exists a 
   canonical isomorphism $i_{\sigma,\tau}:H_c^i(\M_\sigma;\overline{\Q}_\ell)\rightarrow
   H_c^i(\M_\tau;\overline{\Q}_\ell)$ such that the following diagram commutes
 $$
 \xymatrix{H_c^i(\M_\tau;\overline{\Q}_\ell)\ar[rr]^{w^*}&&H_c^i(\M_{w^{-1}\tau w};\overline{\Q}_\ell)\\
   H_c^i(\M_\sigma;\overline{\Q}_\ell)\ar[u]^{i_{\sigma,\tau}}\ar[rr]^{w^*}&&H_c^i(\M_{w^{-1}\sigma w};\overline{\Q}_\ell)
   \ar[u]_{i_{w^{-1}\sigma w,w^{-1}\tau w}} }.
$$
Moreover for all $\sigma,\tau,\zeta\in\t_\v^{\gen}$ we have $i_{\sigma,\tau}\circ i_{\zeta,\sigma}=i_{\zeta,\tau}$.
\label{key1}\end{theorem}

Before writing the proof let us explain the rough strategy.
 Put
$$
\M:=\left\{(\varphi,X,gT_\v,\sigma)\in{\rm
    Rep}_\K(\Gamma,\v)\times\gl_\v\times(\GL_\v/T_\v)\times
  \t_\v^{\gen}\,\left|\,g^{-1}Xg=\sigma,  
    \mu_\v(\varphi)=X\right\}\right./\!/\GL_\v 
$$
where $\GL_\v$ acts on the first three coordinates as before and
trivially on the last one. Denote by $\pi:\M\rightarrow \t_\v^{\gen}$
the projection to the last coordinate. Note that the stalk at $\sigma$
of the sheaf $R^i\pi_!\overline{\Q}_\ell$ is
$H_c^i(\pi^{-1}(\sigma);\overline{\Q}_\ell)=H_c^i(\M_\sigma;\overline{\Q}_\ell)$. Since
$\pi$ commutes with Weyl group actions, to prove Theorem \ref{key1} we
need to prove that the sheaf $R^i\pi_!\overline{\Q}_\ell$ is constant.
Unfortunately we do not known any algebraic proof of this last
statement so we do not have a proof which works independently from
${\rm char}\,\K$. We follow the same strategy as in~\cite[Proof of
Proposition 2.3.1]{crawley-etal} proving first the statement with
$\K=\C$ using the hyperk\"ahler structure on quiver varieties and
then proving the positive characteristic case by reducing modulo
$p$. 

We prefer to work with \'etale $\Z/\ell^n\Z$-sheaves instead of
$\ell$-adic sheaves. We will show that the sheaves
$R^i\pi_!\Z/\ell^n\Z$ are constant for all $n\geq 1$ assuming that
${\rm char}\,\K$ is either sufficiently large or equal to $0$.  This
will prove Theorem \ref{key1} for the \'etale cohomology
$H_c^i(\M_\sigma;\Z/\ell^n\Z)$ with coefficients in $\Z/\ell^n\Z$.  We
then pass to the direct limit to get the statement for $\ell$-adic
cohomology $$H_c^i(\M_\sigma;\overline{\Q}_\ell)=\left(\varprojlim
  H_c^i(\M_\sigma;\Z/\ell^n\Z)\right)\otimes_{\Z_\ell}\overline{\Q}_\ell.$$

\begin{proof}[Proof of Theorem \ref{key1}] 
  (i) Assume that $\K=\C$. From the equivalence of categories between
  constructible \'etale sheaves on a complex variety and constructible
  sheaves on the underlying topological space (see \cite[\S 6]{BBD}
  for more details) we are reduced to prove that if $\Z/\ell^n\Z$ is
  the constant sheaf (for the analytic topology) on the complex
  variety $\M$, then the analytic sheaf $R^i\pi_!\Z/\ell^n \Z$ is
  constant on $\t_\v^{\gen}$. But this is exactly what is proved in
  \cite[Lemma 48]{maffei} where the hyperk\"ahler structure on quiver
  varieties is used as an essential ingredient. 
  
  We note that in \cite{maffei} it is assumed that the quiver does not have loops. The proof of  \cite[Lemma 48]{maffei} however remains valid in the general case. The only 
  non-trivial ingredient is the proof of the surjectivity \cite[CON3 \S 5]{maffei} of the hyperk\"ahler moment map over the regular locus in the case 
  for quivers with possible loops. More recent result \cite[Theorem 2]{crawley-boevey} implies that as long as the dimension vector is a root 
  the complex moment map is surjective for any quiver. By hyperk\"ahler rotation we get that the corresponding hyperk\"ahler moment map is also surjective.

  (ii) The $\C$-schemes $\M, \t_\v^{\gen}$ and the projection $\pi:
  \M\rightarrow\t_\v^{\gen}$ to the last coordinate can actually be
  defined over~$\Z$ (see~\cite[Appendix B]{crawley-etal}). We will
  denote these by $\M/_\Z$, $\t_\v^{\gen}/_\Z $,
  $\pi/_\Z:\M/_\Z\rightarrow\t_\v^{\gen}/_\Z$ and denote by
  $\calF=\calF_\Z$ the sheaf $R^i(\pi/_\Z)_!\Z/\ell^n\Z$.  Recall that
  if $f:X\rightarrow {\rm Spec}\,\Z$ denotes the structure map of a
  $\Z$-scheme, then by Deligne \cite[Theorem 1.9]{deligne}, for any
  constructible $\Z/\ell^n\Z$-sheaf $\mathcal{E}$ on $X$, there exists
  an open dense subset $U$ of ${\rm Spec} \,\Z$ such that for any base
  change $S\rightarrow U\subset {\rm Spec}\,\Z$ we have
  $(f_*\mathcal{E})_S\simeq (f_S)_*(\mathcal{E}_S)$. Denote by $t/_\Z$
  the structure map of $\t_\v^{\gen}/_\Z$. We are thus reduced to
  prove that the canonical map 

$$\eta:(t/_\Z)^*(t/_\Z)_*
  \calF\rightarrow\calF$$
 given by adjointness is an isomorphism over
  an open subset $U$ of ${\rm Spec}\,\Z$. Indeed, this will prove that
  for any prime $p$ such that $p\Z\in U$, the map
  $\eta/_{\overline{\F}_p}$ obtained from $\eta$ by base change is an
  isomorphism which is equivalent to say that the sheaf
  $\calF_{\overline{\F}_q}\simeq
  R^i(\pi/_{\overline{\F}_q})_!\Z/\ell^n\Z$ is constant. By (i) we
  know that the sheaf $\calF_\C$ is constant, i.e., the isomorphism
  $\eta_\C:(t/_\C)^*(t/_\C)_* \calF_\C\rightarrow \calF_\C$ of \'etale
  sheaves obtained from $\eta$ by base change is an isomorphism. Hence
  if $\calK$ and $\calC$ denote respectively the kernel and co-kernel
  of $\eta$, then $\calK_\C=0$ and $\calC_\C=0$. Since by Deligne
  \cite[Theorem 1.9]{deligne} the sheaf $(t/_\Z)^*(t/_\Z)_* \calF$ is
  constructible over an open subset $V$ of ${\rm Spec}\,\Z$, the
  sheaves $\calK$ and $\calC$ are also constructible over $V$ and so
  the support of $\calK_V$ and $\calC_V$ are constructible sets. Since
  $\calK_\C=0$ and $\calC_\C=0$, the supports do not contain the
  generic point and so there exists an open subset $U$ of $V$ such
  that $\calK_U=\calC_U=0$.
\end{proof}

Assume that ${\rm char}\,\K$ is as in Theorem \ref{key1}. For $w\in W_\v$ and $\tau\in\t_\v^{\gen}$  define
$$\rho^i(w):H_c^i(\M_\tau;\overline{\Q}_\ell)\rightarrow
H_c^i(\M_\tau;\overline{\Q}_\ell)$$ as the composition
$i_{w\tau
  w^{-1},\tau}\circ (w^{-1})^*$. 
   The following proposition is a straightforward 
consequence of Theorem \ref{key1}.
\begin{proposition}The map 
$$
\begin{array}{cccc}
\rho^i=\rho^i_\K:\quad & W_\v& \rightarrow &\GL\left(H_c^i(\M_\tau;\overline{\Q}_\ell)\right)\\
&w & \mapsto &\rho^i(w)
\end{array}
$$
 is a representation of $W_\v$ which, 
 does not depend on the choice of $\tau\in \t_\v^{\gen}$.
\label{rho}\end{proposition}

\begin{theorem} Assume that ${\rm char}\,\K\gg0$. Then the $\overline{\Q}_\ell$-representations $\rho^i_\K$ and $\rho^i_\C$ are
  isomorphic. 
\label{compar}\end{theorem}

\begin{proof} Let $U$ be the open subset of ${\rm Spec}\,\Z$ as in the
  proof of Theorem \ref{key1} over which $\eta$ is an isomorphism and
  $(t/_\Z)_*\calF$ is constructible.  Namely, if we denote by $\calE$
  the restriction of $(t/_\Z)_*\calF$ to $U$, then $\calE$ is
  constructible and $\calF_U$ is canonically isomorphic to
  $(t/_\Z)^*\calE$. Since $\calE$ is constructible, it is locally
  constant over an open subset of $U$ and so we may assume (taking an
  open subset of $U$ if necessary) that there is an \'etale covering
  $s:U'\rightarrow U$ over $U$ with $U'$ irreducible such that the
  sheaf $s^*\calE$ is constant. Now let $p$ be a prime such that
  $p\Z\in U$ and assume that ${\rm char}\,\K=p$. Consider the
  geometric points $\overline{x}_0:{\rm Spec}\,\C\rightarrow U$ and
  $\overline{x}_p:{\rm Spec}\,\K\rightarrow U$ induced by the maps
  $\Z\subset\C$ and $\Z\rightarrow\Z/p\Z\rightarrow\K$, and let
  $\overline{x}_0'$ and $\overline{x}_p'$ be geometric points of $U'$
  lying over $\overline{x}_0$ and $\overline{x}_p$ respectively. Since
  the sheaf $s^*\calE$ is constant, we have a canonical isomorphism
  $(s^*\calE)_{\overline{x}_0'}\simeq (s^*\calE)_{\overline{x}_p'}$
  and so we get an isomorphism
  $h:\calE_{\overline{x}_0}\simeq\calE_{\overline{x}_p}$. Since any
  two geometric points $\tau_0:{\rm Spec}\,\C\rightarrow
  \t_\v^{\gen}/_\C\rightarrow\t_\v^{\gen}/_\Z$ and $\tau_p:{\rm
    Spec}\,\K\rightarrow\t_\v^{\gen}/_\K\rightarrow \t_\v^{\gen}/_\Z$
  lie over $\overline{x}_0$ and $\overline{x}_p$ respectively we have
  $\calF_{\tau_0}\simeq\calE_{\overline{x}_0}\simeq
  \calE_{\overline{x}_p}\simeq\calF_{\tau_p}$. Identifying $\tau_0$
  and $\tau_p$ with their respective images in $\t_\v^{\gen}/_\C$ and
  $\t_\v^{\gen}/_\K$, we deduce an isomorphism
  $$\sigma_{\tau_0,\tau_p}:H_c^i(\M_{\tau_0};\overline{\Q}_\ell)=
  \calF_{\tau_0}\simeq\calF_{\tau_p}=
  H_c^i(\M_{\tau_p};\overline{\Q}_\ell).$$If $\tau_0'$ is another
  geometric point ${\rm Spec}\,\C\rightarrow
  \t_\v^{\gen}/_\C\rightarrow\t_\v^{\gen}/_\Z$, then the isomorphism
  $i_{\tau_0,\tau_0'}:\calF_{\tau_0}\simeq\calF_{\tau_0'}$ of Theorem
  \ref{key1} is the composition of the canonical isomorphisms
  $\calF_{\tau_0}\simeq\calE_{\overline{x}_0}\simeq\calF_{\tau_0'}$,
  and same thing with geometric points over ${\rm Spec}\,\K$. The
  following commutative diagram summarizes what we just said

$$
\xymatrix{H_c^i(\M_{\tau_0};\overline{\Q}_\ell)\ar[dd]_{i_{\tau_0,\tau_0'}}
  \ar[dr]\ar[rrr]^{\sigma_{\tau_0,\tau_p}}
  &&&H_c^i(\M_{\tau_p};\overline{\Q}_\ell)\ar[dd]^{i_{\tau_p,\tau_p'}}\ar[dl]\\
  &\calE_{\overline{x}_0}\ar[r]&\calE_{\overline{x}_p}&\\
  H_c^i(\M_{\tau_0'};\overline{\Q}_\ell)\ar[rrr]^{\sigma_{\tau_0',\tau_p'}}\ar[ur]
  &&&H_c^i(\M_{\tau_p'};\overline{\Q}_\ell)\ar[ul]},
$$
where the maps are all isomorphisms. Now the fact that 
$\sigma_{\tau_0,\tau_p}$ commutes also with the
isomorphisms $$w^*_\C:H_c^i(\M_{\tau_0};\overline{\Q}_\ell)\rightarrow
H_c^i(\M_{w^{-1}\tau_0w};\overline{\Q}_\ell)$$ and
$$w^*_\K:H_c^i(\M_{\tau_p};\overline{\Q}_\ell)\rightarrow
H_c^i(\M_{w^{-1}\tau_pw};\overline{\Q}_\ell)$$ follows from the fact
that the action of $W_\v$ is defined over $\Z$. This proves that the
$\overline{\Q}_\ell$-representations $\rho_\C^i$ and $\rho_\K^i$ are isomorphic.
\end{proof}

Note that  the
representation $\rho^i_\C$ of $W_\v$ on 
$H_c^i(\M_\tau;\overline{\Q}_\ell)$ is defined so that it agrees via the
comparison theorem with the action $\varrho^{i}$ of $W_\v$ on the compactly
supported cohomology $H_c^i(\M_\tau;\C)$ as defined from \cite[Lemma 48]{maffei}.

\subsubsection{Introducing Frobenius}

Here $\K$ is an algebraic closure of a finite field $\F_q$. We use the
same letter $F$ to denote the Frobenius endomorphism on ${\rm
  Rep}_\K\left(\Gamma,\v\right)$ and $\gl_\v$ that raises entries of
matrices to their $q$-th power.  Let us first recall how we define a
map from the set of $F$-stable regular semisimple orbits of $\gl_\v$
onto the set of conjugacy classes of the Weyl group $W_\v$ of
$\GL_\v$. 

Take any regular semisimple $X\in\gl_\v(\F_q)=(\gl_\v)^F$. By
definition the centralizer $C_{\GL_\v}(X)$ is an $F$-stable maximal
torus $T_X$ of $\GL_\v$. Since maximal tori are all $\GL_\v$-conjugate
to each other, we may write $T_X=gT_\v g^{-1}$ for some $g\in
\GL_\v$. Applying $F$ to the identity and using the fact that both
$T_\v$ and $T_X$ are $F$-stable we find that $g^{-1}F(g)\in
N_{\GL_\v}(T_\v)$. Taking the image of $g^{-1}F(g)$ in $W_\v$ gives a
well-defined map from the set of regular semisimple elements of
$\gl_\v(\F_q)$ to $W_\v$ and so a map from $F$-stable regular
semisimple orbits of $\gl_\v$ to conjugacy classes of $W_\v$ (as
$F$-stable orbits always contain a representative in
$\gl_\v(\F_q)$). 

Given $w\in W_\v$ choose a generic regular semisimple adjoint orbit
$\calO$ of $\gl_\v(\F_q)$ mapping to $w$ and denote by $\calQ_\tv^w$
the corresponding quiver variety
associated to $\calO$ defined by \eqref{supernovadefined}. If $w=1$ then we will denote these simply by
$\calQ_\tv$ instead of $\calQ_\tv^1$; note that in this case the orbit
$\calO$ has all its eigenvalues in $\F_q$.  Since $\calO$ is
$F$-stable the  quiver variety $\calQ_\tv^w$ inherits an
action of the Frobenius endomorphism which we also denote by $F$.
 
By Lemma \ref{Ms} and Proposition \ref{rho} we have a well-defined representation $\rho^i$ of $W_\v$ in $H_c^i(\calQ_\tv,\overline{\Q}_\ell)$ assuming that the characteristic of $\K$ is large enough.

The aim of this section is to prove the following theorem.
 \begin{theorem} Assume that ${\rm char}\,\K\gg0$. We have 
 $$
\#  \calQ_\tv^w(\F_q)=\sum_i\Trace\left(\rho^{2i}(w),{ H_c^{2i}(\calQ_\tv;\overline{\Q}_\ell)}
\right) q^i. 
$$
\label{Weyl} \end{theorem}
 
We keep the notation introduced in \S \ref{W1}. 
 
 The Frobenius morphism $(\varphi,X,gT_\v)\mapsto
 (F(\varphi),F(X),F(g)T_\v)$ defines a bijective morphism
 $F:\M_\sigma\rightarrow\M_{F(\sigma)}$.  Since the map $\pi:\M\rightarrow\t_\v^{\gen}$ commutes with $F$, the following diagram commutes
 $$
 \xymatrix{H_c^i(\M_\tau;\overline{\Q}_\ell)\ar[rr]^{F^*}&&H_c^i(\M_{F^{-1}(\tau)};\overline{\Q}_\ell)\\
   H_c^i(\M_\sigma;\overline{\Q}_\ell)\ar[u]^{i_{\sigma,\tau}}\ar[rr]^{F^*}&&H_c^i(\M_{F^{-1}(\sigma)};\overline{\Q}_\ell)
   \ar[u]_{i_{F^{-1}(\sigma),F^{-1}(\tau)}} }
$$
for all $\sigma,\tau\in\t_\v^{\gen}$, where $i_{\sigma,\tau}$ is as in Theorem \ref{key1}.

For $w\in W_\v$ consider the $w$-twisted Frobenius endomorphism
$$
\begin{array}{cccc}
wF: \quad &\gl_\v&\rightarrow&\gl_\v \\
& X&\mapsto& \dot{w}F(X)\dot{w}^{-1}
\end{array}
$$
where $\dot{w}$ is a representative of $w$ in
$N_{\GL_\v}(T_\v)$. 

Let $\sigma\in (\t_\v^{\gen})^{wF}$.  Since $wF(\sigma)=\sigma$ we get
a Frobenius endomorphism 
$$
\begin{array}{cccc}
wF:\quad & \M_\sigma& \rightarrow &\M_\sigma\\
& (\varphi,X,gT_\v) &\mapsto &
(F(\varphi),F(X),F(g)\dot{w}^{-1}T_\v).
\end{array}
$$

Let $\tau\in (\t_\v^{\gen})^F$. By Theorem \ref{key1}, the following diagram
commutes
\beq
\xymatrix{H_c^i(\M_\tau;\overline{\Q}_\ell)\ar[rr]^{\rho^i(w)}&&
H_c^i(\M_\tau;\overline{\Q}_\ell)\ar[rr]^{F^*}&&H_c^i(\M_\tau;\overline{\Q}_\ell)\\ 
H_c^i(\M_\sigma;\overline{\Q}_\ell)\ar[rr]^{w^*}\ar[u]^{i_{\sigma,\tau}}&&
H_c^i(\M_{F(\sigma)};\overline{\Q}_\ell)\ar[rr]^{F^*}\ar[u]^{i_{F(\sigma),\tau}}&&
H_c^i(\M_\sigma;\overline{\Q}_\ell)\ar[u]_{i_{\sigma,\tau}}. 
}
\label{w}\eeq
Note that the arrow labelled by $w^*$ is well-defined as $F(\sigma)=
\dot{w}^{-1}\sigma\dot{w}$.

 Applying the Grothendieck
trace formula to $wF:\M_\sigma\rightarrow\M_\sigma$ we find that
\begin{align*}
 \#\M_\sigma(\K)^{wF} &=\sum_i(-1)^i\Trace \left((wF)^*,
   {H_c^i(\M_\sigma;\overline{\Q}_\ell)}\right)\\  
&=\sum_i(-1)^i\Trace\left(F^*\circ \rho^i(w) 
,{H_c^i(\M_\tau;\overline{\Q}_\ell)}\right)\\
&=\sum_i\Trace\left(F^*\circ \rho^{2i}(w) 
,{H_c^{2i}(\M_\tau;\overline{\Q}_\ell)}\right)
\end{align*}
The second identity follows from the diagram~\eqref{w} and the last one from the fact that $\M_\tau$ has vanishing odd cohomology (see Theorem \ref{odd}(ii)).

It is well-known \cite{crawley-etal} that the cohomology of any
generic  quiver variety  defined over $\F_q$ (i.e. with parameters in $\F_q^I$)  is \emph{pure}, in the sense that the
eigenvalues of the Frobenius $F^*$ on the compactly supported $i$-th
cohomology have absolute value $q^{i/2}$. (It is proved in \cite{aha}
\cite{hausel-kac} that over $\C$ the cohomology has pure mixed Hodge
structure.) It is also well-known that these quiver varieties are
\emph{polynomial count}, i.e., there exists a polynomial
$P(T)\in\Q[T]$ such that for any finite field extension $\F_{q^n}$, the evaluation of $P$ at $q^n$ counts the number of points of the variety over $\F_{q^n}$ (see for instance \cite{crawley-etal},
\cite{hausel-kac}, \cite{aha}).  Since $\tau\in (\t_\v^{\gen})^F$, the variety $\M_\tau$ is  thus pure and polynomial count. Hence by \cite[Appendix A]{crawley-etal} we  have the following result.

\begin{theorem}The automorphism $F^*$ on $H_c^{2i}(\M_\tau;\overline{\Q}_\ell)$ has  a
  unique eigenvalue $q^i$. 
\end{theorem}

It is not difficult to verify that the two
automorphisms $\rho^i(w)$ and $F^*$ commute for all $i$ and $w\in
W_\v$. We have
$$
\Trace\left(F^*\circ\rho^{2i}(w),{ 
H_c^{2i}(\M_\tau;\overline{\Q}_\ell)}\right)=\Trace\left(\rho^{2i}(w),{
  H_c^{2i}(\M_\tau;\overline{\Q}_\ell)}\right)q^i,
$$ 
from which we deduce that 
\beq
\#\M_\sigma(\K)^{wF}=\sum_i\Trace\left(\rho^{2i}(w),{
  H_c^{2i}(\M_\tau;\overline{\Q}_\ell)}\right)q^i. 
\label{Weyl1}\eeq

Hence Theorem \ref{Weyl} follows from (\ref{Weyl1}) and the following lemma.

\begin{lemma} Let $w\in W_\v$ and let $\sigma$ be a representative of the orbit $\calO^w$ in $(\t_\v)^{wF}$, then 

$$
\#\M_\sigma(\K)^{wF}=\#\calQ_\tv^w(\K)^F.
$$
\end{lemma}

\begin{proof}It follows from the fact that the isomorphism
  $\M_\sigma\rightarrow \calQ_\tv^w$, $(\varphi,X,gT_\v)\mapsto
  \varphi$ of Lemma~\ref{Ms} commutes with $wF$ and $F$.
\end{proof}

\subsection{Counting points of  $\calQ^w_\tv$ over finite
  fields} 

In this section we will evaluate $\#\calQ^w_\tv(\F_q)$. As in \S
\ref{results} we label the vertices of $\Gamma$ by $\{1,\dots,r\}$ and
we denote by $\calP_\v$ the set of all multi-partitions
$(\mu^1,\dots,\mu^r)$ of size $(v_1,\dots,v_r)$.  The conjugacy class
of an element $w=(w_1,\ldots,w_r)\in W_\v$ determines a 
multi-partition $\lambdahat=(\lambda^1,\dots,\lambda^r)\in\calP_\v$,
where $\lambda^i$ is the cycle type of $w_i\in S_{v_i}$. We will call
$\lambdahat$ the {\it cycle type} of $w$.

Let
 $$
p_\lambdahat:=p_{\lambda^1}(\x_1)\cdots
p_{\lambda^r}(\x_r),
$$
where for a partition $\lambda$, $p_\lambda(\x_i)$ is the
corresponding power symmetric function in the variables of
$\x_i=\{x_{i,1},x_{i,2},\dots\}$ (see \cite[Chapter I,\S 2]{macdonald}). Recall that $\epsilon$ denotes the
sign character of $W_\v$. Denote by $\tilde{\bf C}_\v$  the Cartan matrix of the
quiver $\tilde{\Gamma}_\v$, then
$$
d_\tv:=1-\tfrac 12{^t}\tv \tilde{\bf C}_\v\tv,
$$
equals $\tfrac12 \dim\calQ_\tv$ if $\calQ_\tv$ is non-empty.

The aim of this section is to prove the following theorem.
\begin{theorem} Let $w\in W_\v$ have cycle type
  $\lambdahat\in\calP_\v$. Then
\beq
\#\calQ_\tv^w(\F_q)
=q^{d_\tv}\epsilon(w)\left\langle
  \bH(\x_1,\dots,\x_r;q),p_\lambdahat\right\rangle. 
\eeq
\label{count}\end{theorem}

Fix a non-trivial additive character $\Psi:\F_q\rightarrow\C^\times$
and for $X,Y\in\gl_\v$ put $\langle X,Y\rangle:=\Trace(XY)$. Denote by
$C(\gl_\v)$ the $\C$-vector space of all functions
$\gl_\v^F\rightarrow\C$ and define the Fourier transform $\calF:
C(\gl_\v)\rightarrow C(\gl_\v)$ by
$$
\calF(f)(X):=\sum_{Y\in\gl_\v^F}\Psi\left(\langle
  X,Y\rangle\right)f(Y),
$$
with $f\in C(\gl_\v)$ and $X\in\gl_\v^F$. Basic properties of $\calF$
can be found for instance in \cite{letellier}. For an $F$-stable
adjoint orbit $\calO$ of $\gl_\v$ we denote by $1_\calO\in C(\gl_\v)$
the characteristic function of the adjoint orbit $\calO^F$ of
$\gl_\v^F$, i.e., $1_\calO(X)=1$ if $X\in\calO^F$ and $1_\calO(X)=0$
if $X\notin\calO^F$.

\begin{proposition} For any $F$-stable generic adjoint orbit $\calO$
  of $\gl_\v^0$ we have 
\begin{align}
\nonumber
\#\left(\mu_\v^{-1}(\calO)/\!/\GL_\v\right)^F&=
\frac{q-1}{|\GL_\v^F|}\,\,\#\mu_\v^{-1}(\calO)^F \\
&=\frac{(q-1)\,|{\rm
    Rep}_{\F_q}(\Gamma,\v)|}{|\GL_\v^F|\cdot|\gl_\v^F|}\sum_{X\in\gl_\v^F}
\#\left\{\left.\varphi\in {\rm
      Rep}_{\F_q}(\Gamma,\v)\,\right|\,[X,\varphi]=0\right\}
\calF(1_\calO)(X)
\label{form}
\end{align}
where $[X,\varphi]=0$ means that for each arrow $\gamma=i\rightarrow
j$ in $\Omega$ we have $X_j\varphi_\gamma=\varphi_\gamma X_i$.
\end{proposition}
\begin{proof} The first equality comes from the fact that $G_\v
  =\GL_\v/\bG_m$ is connected and acts freely on $\mu_\v^{-1}(\calO)$
  (see Theorem \ref{odd}(i)). For the second write
$$
\#\,\mu_\v^{-1}(\calO)^F=\sum_{z\in\calO^F}\#\,\mu_\v^{-1}(z)^F
$$
and use \cite[Proposition 2]{hausel-kac}.
\end{proof}

In order to compute the right hand side of Formula (\ref{form}) we
need to introduce some notation. Denote by $\mathfrak{O}$ the set of
all $F$-orbits of $\K$. The adjoint orbits of $\gl_n^F$ are
parametrized by the maps $h:\mathfrak{O}\rightarrow\calP$ such that
$$
\sum_{\gamma\in\mathfrak{O}}|\gamma|\cdot |h(\gamma)|=n.
$$
Denote by $0\in\calP$ the unique partition of $0$. The \emph{type} of
an adjoint orbit $\calO$ of $\gl_n^F$ corresponding to
$h:\mathfrak{O}\rightarrow\calP$ is defined as the map $\omega_\calO$
that assigns to a positive integer $d$ and a non-zero partition
$\lambda$ the number of Frobenius orbits $\gamma\in \mathfrak O$ of
degree $d$ such that $h(\gamma)=\lambda$.

It is sometimes also convenient (see \cite{aha})
to write a type as follows. Choose a total ordering $\geq$ on
partitions which we extend to a total ordering on the set
$\Z_{>0}\times(\calP\backslash\{0\})$ as $(d,\lambda)\geq
(d',\lambda')$ if $d\geq d'$ and $\lambda\geq\lambda'$. Then we may
write the type $\omega_\calO$ as a the strictly decreasing sequence
$(d_1,\lambda^1)^{n_1}(d_2,\lambda^2)^{n_2}\cdots(d_s,\lambda^s)^{n_s}$
with $n_i=\omega_\calO(d_i,\lambda^i)$. The set of all
non-increasing sequences
$(d_1,\lambda^1)(d_2,\lambda^2)\cdots(d_s,\lambda^s)$ of size $n$
(i.e., $\sum_{i=1}^sd_i|\lambda^i|=n$) denoted by $\T_n$ parametrizes
the types of the adjoint obits of $\gl_n^F$.

It is easy to extend this to adjoint orbits of $\gl_\v^F$. They are
parametrized by the set of all maps
$h=(h_1,\dots,h_r):\mathfrak{O}\rightarrow\calP^r$ such that for each
$i=1,\dots,r$, we have
$$
\sum_{\gamma\in\mathfrak{O}}|\gamma|\cdot |h_i(\gamma)|=v_i.
$$
A \emph{type} of  an adjoint orbit  $\calO$ of $\gl_\v^F$
corresponding to $h:\mathfrak{O}\rightarrow\calP^r$ is now a map $\omhat_\calO$ that assigns to a positive integer $d$ and a
non-zero multi-partition $\lambdahat=(\lambda^1,\ldots,\lambda^r)$ the number of Frobenius orbits $\gamma\in\mathfrak{O}$ of degree $d$ such that $h(\gamma)=\lambdahat$.

As above, after choosing a total ordering on the set
$\Z_{>0}\times(\calP^r\backslash\{0\})$ we may write $\omhat_\calO$ as
a (strictly) decreasing sequence
$(d_1,\lambdahat_1)^{n_1}(d_2,\lambdahat_2)^{n_2}\cdots(d_s,\lambdahat_s)^{n_s}$
with $n_i=\omhat_\calO(d_i,\lambdahat_i)$. We denote by $\T_\v$ the
set of all non-increasing sequences
$(d_1,\lambdahat_1)\dots(d_s,\lambdahat_s)$ of size $\v$ so that
$\T_\v$ parametrizes the types of the adjoint orbits of $\gl_\v^F$.
We may also write a type $\omhat\in\T_\v$ as
$(\omega_1,\ldots,\omega_r)$, where
$\omega_i=(d_1,\lambda^{i,1})(d_2,\lambda^{i,2})\ldots$, with
$\lambda^{i,j}$ the $i$-th coordinate of $\lambdahat_j$, is a type in
$\T_{v_i}$.

Given any family $\{A_\muhat(\x_1,\dots,\x_r;q)\}_{\muhat\in\calP^r}$
of functions separately symmetric in each set $\x_1,\dots,\x_k$ of
infinitely many variables with $A_0=1$, we extend its definition to
types $\omhat=(d_1,\lambdahat_1)\dots(d_s,\lambdahat_s)\in\T_\v$ as
 $$
 A_\omhat(\x_1,\dots\x_r;q):=
 \prod_{i=1}^sA_{\lambdahat_i}(\x_1^{d_i},\dots,\x_r^{d_i};q^{d_i}),   
 $$
 where $\x^d$ stands for all the variables $x_1,x_2,\dots$ in $\x$
 replaced by $x_1^d,x_2^d,\dots$. 

For $\pihat=(\pi^1,\dots,\pi^r)\in\calP^r$ put 
$$
\calA_\pihat(q)=\prod_{i\rightarrow j\in \Omega} q^{\langle
  \pi^i,\pi^j\rangle},
\qquad \calZ_\pihat(q):=\prod_{i\in I}
q^{\langle\pi^i,\pi^i\rangle} 
\prod_{k}\prod_{j=1}^{m_k(\pi^i)}(1-q^{-j}),
\qquad \calH_\pihat(q):=\frac{\calA_\pihat(q)}{\calZ_\pihat(q)},
$$
where we use the same notation as in \S \ref{results}.  Then by
\cite[Theorem 3.4]{hua}, for any element $X$ in an adjoint orbit of
$\gl_\v^F$ of type $\omhat\in\T_\v$ we have
$$ 
\#\left\{\left.\varphi\in {\rm
      Rep}_{\F_q}(\Gamma,\v)\,\right|\,[X,\varphi]=0\right\}=\calA_\omhat(q),
\qquad
|C_{\GL_\v^F}(X)|=\calZ_\omhat(q).
$$

Hence
\beq
\frac{1}{|\GL_\v^F|}\sum_{X\in\gl_\v^F} \#\left\{\left.\varphi\in {\rm
      Rep}_{\F_q}(\Gamma,\v)\,\right|\,[X,\varphi]=0\right\}
\calF(1_\calO)(X)=\sum_{\omhat\in\T_\v}\calH_\omhat(q)\sum_{\calO'}
\calF(1_\calO)(\calO'),   
\label{hua1}
\eeq 
where the last sum is over the adjoint orbits $\calO'$ of $\gl_\v^F$
of type $\omhat$ and $\calF(1_\calO)(\calO')$ denotes the common value
$\calF(1_\calO)(X)$ for $X\in \calO'$.

For a type $\omhat=(d_1,\lambdahat_1)\cdots(d_s,\lambdahat_s)\in\T_\v$  put
$$
C_\omhat^o:=\begin{cases}
  \frac{\mu(d)}{d}(-1)^{s-1}\frac{(s-1)!}{\prod_\lambdahat
    m_{d,\lambdahat}(\omhat)!} &\text{ if } d_1=d_2=\cdots=d_s=d\\ 
0&\text{ otherwise}.\end{cases}
$$
where $m_{d,\lambdahat}(\omhat)$ denotes the multiplicity of the pair
$(d,\lambdahat)$ in $\omhat$ and where $\mu$ denotes the ordinary
M\"obius function. 

Recall that we defined a map (see the beginning
of~\S\ref{Weyl-gp-action}) from regular semisimple adjoint orbits of
$\gl_\v(\F_q)$ to $W_\v$.

\begin{proposition}Let $w=(w_1,\ldots,w_r)\in W_\v$ have cycle type
  $\lambdahat=(\lambda^1,\dots,\lambda^r)\in\calP_\v$. Let $\calO$ be
  a generic regular semisimple adjoint orbit of $\gl_\v(\F_q)$ mapping
  to $w$ and $\omhat=(\omega_1,\ldots,\omega_r)$ is any type in
  $\T_\v$. Then
$$
\sum_{\calO'}
\calF(1_\calO)(\calO')=\epsilon(w)\,
q^{1+\delta_\v/2}
\,C_\omhat^o\prod_{i=1}^r\left\langle
  \tilde{H}_{\omega_i}(\x_i;q),p_{\lambda^i}(\x_i)\right\rangle,
$$
where the sum is over the adjoint orbits $\calO'$ of $\gl_\v^F$
of type $\omhat$, $\calF(1_\calO)(\calO')$ denotes the common value
$\calF(1_\calO)(X)$ for $X\in \calO'$ and $\delta_\v={\rm dim}\,
\GL_\v-{\rm dim}\, T_\v=\sum_i v_i^2-\sum_iv_i$. 
\label{charsum}
\end{proposition}

\begin{proof} For a partition $\lambda=(\lambda_1,\dots,\lambda_m)$
  and $d\in\Z_{>0}$, define
  $d\cdot\lambda:=(d\lambda_1,\dots,d\lambda_m)$, and for a type
  $\tau=(d_1,\tau^1)\cdots( d_s,\tau^s)\in\T_n$ put
  $[\tau]:=\cup_id_i\cdot\tau^i$, a partition of $n$.  We will say
  that two types $\nu=(d_1,\nu^1)\cdots(d_s,\nu^s)$ and
  $\omega=(e_1,\omega^1)\cdots(e_l,\omega^l)$ are {\it compatible},
  which we denote $\nu\sim\omega$, if $s=l$ and for each $i=1,\dots,s$
  we have $d_i=e_i$ and $|\nu^i|=|\omega^i|$. For two partitions of
  same size $\lambda,\mu$ let $Q_\mu^\lambda(q)$ be the Green
  polynomial as defined for instance in \cite[Chap. III \S
  7]{macdonald}. For two compatible types $\nu$ and $\omega$ we put
  $Q_\nu^\omega(q):=\prod_i Q_{\nu^i}^{\omega^i}(q^{d_i})$ and
  $Q_\nu^\omega(q):=0$ otherwise.  Let $z_\lambda$ be the order of the
  centralizer of an element of cycle type $\lambda$ in
  $S_{|\lambda|}$. For a type $\nu=(d_1,\nu_1),\ldots,(d_r,\nu_r)$ set
  $z_\nu=\prod_i z_{\nu^i}$.

  Notice that for $\lambda$ a partition $(d_1,d_2,\dots,d_s)$ of $n$
  we have $p_\lambda(\x)=s_\tau(\x)$ where $\tau\in\T_n$ is the type
  $(d_1,1)(d_2,1)\cdots(d_s,1)$. Hence by \cite[Lemma 2.3.5]{aha} for
  any $\omega\in\T_n$ and any partition $\lambda$ of $n$
$$
\left\langle\tilde{H}_{\omega}(\x;q),
  p_{\lambda}(\x)\right\rangle=z_\lambda\sum_{\{\nu\in\T_n\;|\;[\nu]  
  =\lambda\}}\frac{Q^\omega_\nu(q)}{z_\nu}.
$$

We are therefore reduced to prove that
\beq
\sum_{\calO'}
\calF(1_\calO)(\calO')=\epsilon(w)\,
q^{1+\delta_\v/2}
\,C_\omhat^o\prod_{i=1}^rz_{\lambda^i}
\sum_{\{\nu\in\T_n\;|\;[\nu]=\lambda^i\}}
\frac{Q^{\omega_i}_\nu(q)}{z_\nu}. 
\label{reduce}\eeq
The proof of this formula is similar to that of \cite[Theorem
4.3.1(2)]{aha} although the context is different and will require some
new calculations. Embed $\GL_\v$ in $\GL_N$ with
$N=\sum_{i=1}^rv_i$. Write
$\omhat=(d_1,\muhat_1)\cdots(d_s,\muhat_s)$ and define
$\tomega=(d_1,\talpha_1)\cdots(d_s,\talpha_s)\in\T_N$ where for a
multi-partition $\muhat=(\mu^1,\dots,\mu^r)\in\calP^r$, we
put $\talpha=\cup_i\mu^i$. If $\calO'$ is an $F$-stable adjoint
orbit of $\gl_\v$ of type $\omhat$, then the unique $\GL_N$-adjoint
orbit of $\gl_N$ which contains $\calO'$ is of type $\tomega$. Consider a representative of an adjoint orbit of $\gl_\v^F\subset\gl_N^F$ of type
$\omhat$ with Jordan form $\sigma+u$ where $\sigma$ is semisimple and
$u$ is nilpotent. Put $L:=C_{\GL_N}(\sigma)$ and, denote by
$\mathfrak{l}$ the Lie algebra of $L$ and by $z_\mathfrak{l}$ the
center of $\mathfrak{l}$. Note that $\mathfrak{l}$ is not contained in $\gl_\v$ unless each for each $i$ the multi-partition $\muhat_i$ has a unique non-zero coordinate in which case $L=M:=C_{\GL_\v}(\sigma)$. However we always have $z_{\mathfrak{l}}\subset\gl_\v$. Put $(z_\mathfrak{l})_{\reg}:=\{y\in
z_\mathfrak{l}\;|\; C_{\GL_N}(y)=L\}$.  The map that sends $z\in
(z_\mathfrak{l})_{\reg}^F$ to the $\GL_\v$-orbits of $z+u$ surjects
onto the set of adjoint orbits $\calO'$ of $\gl_\v^F$ of type
$\omhat$. The fibers of this map can be identified with
$\{g\in\GL_\v^F\;|\; gLg^{-1}=L,g C_u g^{-1}=C_u\}/M$, where $C_u$ is the $M$-orbit of $u$. Hence we
may in turn identify these fibers  with
$$
W(\omhat):=\prod_{d,\lambdahat}
(\Z/d\Z)^{m_{d,\lambdahat}(\omhat)}\times
S_{m_{d,\lambdahat}(\omhat)}. 
$$
It follows that
\begin{align*}
  \sum_{\calO'}\calF(1_\calO)(\calO')
&=\frac{1}{|W(\omhat)|}
\sum_{z\in(z_\mathfrak{l})_{\reg}^F}\calF(1_\calO)(z+u)\\ 
  &=\frac{1}{|W(\omhat)|}\sum_{z\in(z_\mathfrak{l})_{\reg}^F}
  \prod_{i=1}^r\calF^{\gl_{v_i}}\left(1_{\calO_i}\right)(z_i+u_i),
\end{align*}
where $\calF^{\gl_{v_i}}$ denotes the Fourier transform on
$\gl_{v_i}^F$, $\calO_i$ is the $i$-th coordinate of $\calO$ (a
$\GL_{v_i}^F$-orbit of $\gl_{v_i}^F$) and $z_i,u_i$ are the $i$-th
coordinates of $z,u$ respectively.

It is known \cite[Theorem 7.3.3]{letellier}\cite[Formulas (2.5.4),
(2.5.5)]{aha} that
$$
\calF^{\gl_{v_i}} \left(1_{\calO_i}\right)(z_i+u_i)=\epsilon(w_i)
q^{\frac{1}{2}(v_i^2-v_i)}|M_i^F|^{-1}\sum_{\{h\in\GL_{v_i}^F\;|\;
  h^{-1}z_i
  h\in\mathfrak{t}_{\lambda^i}\}}Q_{hT_{\lambda^i}h^{-1}}^{M_i}(u_i+1)
\Psi\left(\left\langle
    X_i,h^{-1}z_i h\right\rangle\right) 
$$
where $X_i$ is a fixed element in $\calO_i^F$, $T_{\lambda^i}$ is the
unique $F$-stable maximal torus of $\GL_{v_i}$ whose Lie algebra
$\mathfrak{t}_{\lambda^i}$ contains $X_i$, $M_i=C_{\GL_\v}(z_i)$ and
where $Q_{hT_{\lambda^i}h^{-1}}^{M_i}$ is the Green function defined
by Deligne and Lusztig \cite{DLu} (the values of these functions are
products of usual Green polynomials).

It follows that
\begin{align*}
  \sum_{\calO'}\calF(1_\calO)(\calO')
  &=\frac{1}{|W(\omhat)|}\epsilon(w) 
  q^{\frac{1}{2}(\sum_i v_i^2-\sum_i    v_i)}
\sum_{h=(h_1,\dots,h_r)}
\Phi_h(u)
\sum_{z\in(z_\mathfrak{l})_{\reg}^F}
  \prod_{i=1}^r\Psi\left(\left\langle  
      X_i,h_i^{-1}z_ih_i\right\rangle\right)\\ 
  &=\frac{1}{|W(\omhat)|}\epsilon(w) q^{\frac{1}{2}(\sum_i
    v_i^2-\sum_i    v_i)}
\sum_{h}
\Phi_h(u)
  \sum_{z\in(z_\mathfrak{l})_{\reg}^F}\Psi\left(\left\langle  
      X,h^{-1}zh\right\rangle\right)
\end{align*} where 
$h=(h_1,\dots,h_r)$ runs over the set 
$$
\{h\in\GL_\v^F\;|\;
hT_\lambdahat h^{-1}\subset M\} =\{h\in\GL_\v^F\;|\; z\in
hT_\lambdahat h^{-1}\}
$$
with $T_\lambdahat:=\prod_{i=1}^rT_{\lambda^i}$,
$X:=(X_i)_{i=1,\dots,r}\in\mathfrak{t}_\lambdahat\cap\calO$ and where
to simplify we set
$$
\Phi_h(u):=
  \prod_{i=1}^r|M_i^F|^{-1}
    Q_{h_iT_{\lambda^i}h_i^{-1}}^{M_i}(u_i+1).
$$

To finish the proof it suffices to check the following two formulas
\begin{align*}
  &\sum_{z\in(z_\mathfrak{l})_{\reg}^F}\Psi \left(\left\langle
      X,h^{-1}zh\right\rangle\right)=
  \begin{cases}
(-1)^{s-1}\mu(d)q(s-1)!
    & \text{ if } d_i=d \text{ for all }\; i=1,\dots,s\\
    0&\text{ otherwise,}
\end{cases}\label{1} \\
  &\sum_h \Phi_h(u)=
  \prod_{i=1}^rz_{\lambda^i}\sum_{\{\nu\in\T_n\;|\;[\nu]=\lambda^i\}}
  \frac{Q^{\omega_i}_\nu(q)}{z_\nu},
\end{align*}
where recall that $\omhat=(d_1,\muhat^1),(d_2,\muhat^2),\ldots$.  

The proof of the second formula is contained in the proof of
\cite[Theorem 4.3.1(2)]{aha}. For the first formula, by
\cite[Proposition 6.8.3]{letellier2}, it is enough to prove that the
linear character
$\Theta:\mathfrak{t}_\lambdahat^F\rightarrow\C^\times$, $z\mapsto
\Psi\left(\left\langle X,z\right\rangle\right)$ is a \emph{generic}
character, i.e., the restriction of $\Theta$ to $z_{\gl_N}^F$ is
trivial and for any $F$-stable Levi subgroup $L$ (of some parabolic
subgroup) of $\GL_N$ which contains $T_\lambdahat$ the restriction of
$\Theta$ to $z_\mathfrak{l}^F\subset\mathfrak{t}_\lambda^F$ is
non-trivial unless $L=\GL_N$.

But $\Theta$ is generic because the adjoint orbit $\calO$ is
generic. Indeed, since $\Trace(\calO)=0$ we have
$\Theta|_{z_{\gl_N}^F}=1$. Now assume that $L\supset T_\lambdahat$
satisfies $\Theta|_{z_\mathfrak{l}^F}=1$. There is a decomposition
$\K^N=W_1\oplus W_2\oplus\cdots \oplus W_s$, with $W_j\neq 0$,  such that $\mathfrak{l}$
is $\GL_N$-conjugate to $\bigoplus_i\gl(W_i)$.  An element $z\in
z_{\mathfrak{l}}$ is of the form $(\xi_1\iota_1,\dots,\xi_s\iota_s)$
where $\xi_1,\dots,\xi_s\in\K$ and where $\iota_j$ denotes the identity
endomorphism of $W_j$. Denote by $X^j$ the $\gl(W_j)$ coordinate of
$X$. Since $\Theta|_{z_\mathfrak{l}^F}=1$, we must have $\left\langle
  X,z\right\rangle=\sum_{j=1}^s\xi_j\Trace(X^j)=0$ for all
$z=(\xi_1\iota_1,\ldots,\xi_s\iota_s)$ and so
$\Trace(X^j)=0$ for all $j=1,\ldots,s$. Now $\gl_\v$ and
$\mathfrak{l}$ are two Levi sub-algebras of $\gl_N$ that contains
$\mathfrak{t}_\lambdahat$, hence
$\gl_\v\cap\mathfrak{l}\simeq\bigoplus_{i,j}\gl(U_{i,j})$ where
$W_j=\bigoplus_i U_{i,j}$. For each $j=1,\ldots s$, the space
$\bigoplus_i U_{i,j}$ is also a graded subspace of
$\K^N=\K^\v$ on which $X$ acts by $X^j$ and so
by the genericity assumption we must have $W_j=\K^N$ i.e. $L=\GL_N$.

\end{proof}

\begin{proof}[Proof of Theorem \ref{count}]By
  definition~\eqref{dim-defn} 
\begin{align*}
d_\tv&= \sum_{i\rightarrow j\in\Omega}v_iv_j-\sum_i
v_i^2+\delta_\v+2\\ 
&=\dim\Rep_\K(\Gamma,\v)-\dim\gl_\v
  +1+\delta_\v/2. 
\end{align*}

Hence applying Formula (\ref{hua1}) and Proposition \ref{charsum} we
find that
\begin{align*}
  \#\,\calQ_\tv^w(\F_q)&=\epsilon(w)(q-1)
  q^{d_\tv}\left\langle\sum_{\omhat\in\T_\v}
    C_\omhat^o\calH_\omhat(q)\prod_{i=1}^r  
    \tilde{H}_{\omega_i}(\x_i;q),p_\lambdahat\right\rangle\\
  &=\epsilon(w)(q-1) q^{d_\tv}
  \left\langle\Log\,\left(\sum_{\pihat=(\pi_1,\dots\pi_r)\in\calP^r}
      \calH_\pihat(q)\prod_{i=1}^r 
      \tilde{H}_{\pi_i}(\x_i;q)\right),p_\lambdahat\right\rangle\\ 
\end{align*}
The last identity follows from the general properties of $\Log$ (see
\cite[Formula (2.3.9)]{aha}).

\end{proof}

\subsection{\Apols}

We prove a general fact about how to extract \Apols of
the quivers $\Gamma_\muhat$ from the generating function~$\bH(\x_1,\ldots,x_r;q)$
defined in~\eqref{maingen}.  The statement and proof are along the
same lines as~\cite[Theorem 3.2.7]{aha2}. 
For any
multi-partition $\muhat\in\calP^r$ denote by $A_\muhat(q)$ the \Apol
 associated with $(\Gamma_\muhat,\v_\muhat)$ where
$\v_\muhat$ is defined as in \S \ref{results} with $\v=|\muhat|$. For
a partition $\lambda$, denote by $h_\lambda$ the complete symmetric
function as in \cite{macdonald}. 

\begin{theorem} For any $\muhat\in\calP_\muhat$, we have 
$$
\left\langle \bH(\x_1,\dots,\x_r;q),h_\muhat\right\rangle =A_\muhat(q).
$$
\label{theoaha2}
\end{theorem}

\begin{proof} 
Denote by
  $\Omega_\muhat$ the arrows of $\Gamma_\muhat$. The starting idea is
  that for any indecomposable representation $\varphi$ of
  $(\Gamma_\muhat,\v_\muhat)$, the coordinate $\varphi_{i\rightarrow
    j}$ of $\varphi$ at any arrow $i\rightarrow
  j\in\Omega_\muhat\backslash\Omega$ (i.e., $i\rightarrow j$ is an
  arrow on one of the added leg) must be injective (see \cite[Lemma
  3.2.1]{aha2}). Denote thus by ${\rm
    Rep}_\K\left(\Gamma_\muhat,\v_\muhat\right)^*$ the space of all
  representations of $\Gamma_\muhat$ of dimension $\v_\muhat$ whose
  coordinates at the legs are all injective. For a partition
  $\mu=(\mu_1,\mu_2,\dots,\mu_r)$ of $n$, denote by $\calF_\mu$ the
  set of partial flag of $\K$-vector spaces

$$
\{0\}\subset E^{r-1}\subset\cdots\subset E^1\subset E^0=\K^n,
$$
with ${\rm dim}\, E^i-{\rm dim}\, E^{i+1}=\mu_{i+1}$. Then we have a
natural bijection from the isomorphism classes of ${\rm
  Rep}_\K\left(\Gamma_\muhat,\v_\muhat\right)^*$ onto the orbit space

$$
\mathfrak{G}_\muhat:=\left.\left({\rm
      Rep}_\K\left(\Gamma,|\muhat|\right)
    \times\prod_{i=1}^r\calF_{\mu^i}\right)\right/ \GL_\v.
$$
Put $G_\muhat(q):=\#\mathfrak{G}_\muhat(\F_q)$. Then as in
\cite[Theorem 3.2.3]{aha2} we prove that

$$
\Log\left(\sum_{\muhat\in\calP^r}
  G_\muhat(q)m_\muhat\right)
=\sum_{\muhat\in\calP^r\backslash\{0\}}A_\muhat(q)m_\muhat,  
$$
where $m_\muhat$ denotes the monomial symmetric functions as in
\cite{macdonald}. Recall that the basis $\{m_\muhat\}_\muhat$ is dual
to $\{h_\muhat\}_\muhat$ with respect to the Hall pairing. The
analogue of Proposition \cite[Proposition 3.2.5]{aha2} reads \beq
\sum_\muhat G_\muhat(q)=\prod_{d=1}^\infty
\Omega(\x_1^d,\dots,\x_r^d;q^d)^{\phi_d(q)},
\label{aha2}\eeq
where $\Omega(\x_1,\dots,\x_r;q)$ is the series inside the brackets of
Formula (\ref{maingen}),
i.e. $$\Omega(\x_1,\dots,\x_r;q)=\Exp((q-1)^{-1}\bH(\x_1,\dots,\x_r;q))$$
and where $\phi_d(q)$ is the number of Frobenius orbits of
$\overline{\F}_q\backslash\{0\}$ of size $d$. Taking Log on both side
of Formula (\ref{aha2}) and using the properties of Log \cite[Lemma
2.1.2]{aha2} gives

$$
\sum_\muhat A_\muhat(q) m_\muhat=\bH(\x_1,\dots,\x_r;q),
$$
hence the result.
\end{proof}

\subsection{Proof of Theorem~\ref{main}}

\begin{proof}[Proof of Theorem \ref{maintheo}]

We start by proving (i).  Let us denote here by $\calQ_\tv/_\C$ and
$\calQ_\tv/_{\overline{\F}_q}$ the associated 
quiver varieties over the
indicated field. Assume also that the characteristic of $\F_q$ is
large enough so that the results of \S \ref{Weyl-gp-action} apply. Combining
Theorem \ref{Weyl} and Theorem \ref{count} we find that \beq
\left\langle \bH(\x_1,\dots,\x_r;q),p_\lambdahat\right\rangle=
\epsilon(w)\sum_i\Trace\left(\rho^{2i}(w),{
    H_c^{2i}\left(\calQ_\tv/_{\overline{\F}_q};
      \overline{\Q}_\ell\right)}\right) q^{i-d_\tv}.
\label{cc}\eeq

We deduce from Theorem \ref{compar} and the comment below the proof of Theorem \ref{compar} that 
\beq
\left\langle \bH(\x_1,\dots,\x_r;t),p_\lambdahat\right\rangle=
\epsilon(w)\sum_i\Trace\left(\varrho^{2i}(w),{
  H_c^{2i}(\calQ_\tv/_\C;\C)}\right) t^{i-d_\tv} 
\eeq
is an indentity in $\Q[t]$.

The Schur functions $s_\muhat$, with $\muhat\in\calP_\v$, decompose
into power symmetric functions as \cite[Chapter I, Proof of
(7.6)]{macdonald}
$$
s_\muhat=\sum_{\lambdahat\in\calP_\v}\chi^{\muhat}_\lambdahat
p_\lambdahat
$$
where $\chi^\muhat_\lambdahat=\chi^\muhat(w)$ is the value of the
irreducible character $\chi^\muhat$ of $W_\v$ at an element $w\in
W_\v$ of cycle type $\lambdahat$. Hence
\begin{align*}
\left\langle \bH(\x_1,\dots,\x_r;t),s_\muhat\right\rangle &=
q^{-d_\tv}\sum_i\left\langle\chi^\muhat\otimes\epsilon,\varrho^{2i}
\right\rangle_{W_\v} 
t^i\\ 
&=t^{-d_\tv}\sum_i\left\langle\chi^{\muhat'},\varrho^{2i}\right\rangle_{W_\v} t^i,
\end{align*}
and Theorem \ref{maintheo} (i) follows.  Note that $\bH_\muhat^s(t)$ is
a polynomial since $H_c^i(\calQ_{\tilde{\vv}};\C)=0$ unless
$d_\tv\leq i\leq 2d_\tv$ as the variety $\calQ_\tv$ is affine.

We now proceed with the proof of (ii).

Consider the partial
  ordering $\unlhd$ on partitions defined as $\lambda\unlhd\mu$ if for
  all $i$ we have $\sum_i\lambda_i\leq\sum_i\mu_i$. Extend this
  ordering on multi-partitions by declaring that
  $\alphahat\unlhd\betahat$ if and only if for all $i$, we have
  $\alpha^i\unlhd\beta^i$.  A simple calculation shows that if
  $\alphahat\unlhd\betahat$ and $\alphahat\neq\betahat$ for any two
  multi-partitions in $\calP^r$, then $d_\betahat <d_\alphahat$.

Using the relations between Schur functions and complete symmetric
functions \cite[page 101]{macdonald} together with Theorem
\ref{theoaha2} we find that  

\beq
A_\lambdahat(q)=\sum_{\muhat\unrhd\lambdahat}K_{\lambdahat\muhat}'\Hs_\muhat(q),\hspace{.5cm}
\Hs_\muhat(q)=
\sum_{\lambdahat\unrhd\muhat}K_{\muhat\lambdahat}^*A_\lambdahat(q),  
\label{form5}\eeq
where $K=(K_{\lambdahat\muhat})$ is the matrix of Kostka numbers, $K'$
and $K^*$ are respectively the transpose and the transpose inverse of
$K$. Recall \cite[\S 1.15]{kac2} that, if non-zero, $A_\muhat(q)$ is
monic of degree $d_\muhat$, and $A_\muhat(q)$ is non-zero if and only
if $\v_\muhat$ is a root of $\Gamma_\muhat$, with $A_\muhat(q)=1$ if
and only if $\v_\muhat$ is real \cite[\S 1.10]{kac2}. Since
$K_{\muhat\muhat}^*=1$ and since the polynomials $A_\lambdahat(q)$,
with $\lambdahat\unrhd\muhat$, $\lambdahat\neq\muhat$, are of degree
strictly less than $ d_\muhat$, we deduce from the second equality
(\ref{form5}) that if $\v_\muhat$ is a root then $\Hs_\muhat(q)$ is
monic of degree $d_\muhat$. Conversely if $\Hs_\muhat(q)$ is non-zero
then by the first formula (\ref{form5}) the polynomial $A_\muhat(q)$
must be non-zero (i.e. $\v_\muhat$ is a root) as the Kostka numbers
are non-negative, $K_{\muhat\muhat}=1$ and $\Hs_\muhat(q)$ has
non-negative coefficients. This completes the proof.
\end{proof}

\begin{remark}
\label{remark}
1.  When the dimension vector $\v$ is indivisible we can prove Theorem
\ref{main} using the ideas in \cite{letellier2} based on the theory of
perverse sheaves (although the context in \cite{letellier2} is
different). More precisely when $\v$ is indivisible and
$\muhat=(\mu^1,\dots,\mu^r)\in\calP_\v$ there exists a generic adjoint
orbit $\calO$ of $\gl_\v(\C)$ (see \S\ref{supernova}) such that each
component $\calO_i$ of $\calO$ in $\gl_{v_i}$ is of the form
$\zeta_i\cdot I_{v_i}+N_i$ where $\zeta_i\in\C$ and $N_i$ is nilpotent
with Jordan form given by the dual partition $(\mu^i)'$ of
$\mu^i$. Consider the associated singular complex quiver variety
$\calQ_{\v_\mu}:=\mu_\v^{-1}(\overline{\calO})/\!/\GL_\v$ (where
$\mu_\v$ is the moment map defined in \S \ref{supernova}). Then
following the strategy in \cite[Corollary 7.3.5]{letellier2} we can
prove that
 $$
\Hs_\muhat(q)= \sum \dim
  \left(IH_c^{2i}(\calQ_{\v_\muhat};\C)\right)
  q^{i-d_{\v_{\muhat}}}
 $$ 
 is up to a power of $q$ the Poincar\'e polynomial of
 $\calQ_{\v_\muhat}$ for the compactly supported intersection
 cohomology.

2. While in this paper we construct the action of $W_\v$ on
 $H_c^i(\calQ_\tv;\C)$ using the hyperk\"ahler structure on quiver
 varieties, in the case where $\v$ is indivisible it is possible to
 give an alternative construction of this Weyl group action as in
 \cite{letellier2} based on the theory of perverse sheaves and then
 give an alternative proof of Theorem~\ref{main}.
\end{remark}

\section{DT-invariants for symmetric quivers}

\subsection{Preliminaries}
\label{u}

Denote by $\Lambda$ the ring of symmetric functions in the variables
$\x=\{x_1,x_2,\dots\}$ with coefficients in $\Q(q)$ and $\Lambda_n$
those functions in $\Lambda$ homogeneous of degree $n$. We define the
$u$-specialization of symmetric functions as the ring homomorphism
$\Lambda \rightarrow \Q[u]$ that on power sums behaves as follows
$$
p_r(\x)\mapsto 1-u^r.
$$
In plethystic notation this is denoted by $f\mapsto f[1-u]$. 

Note that for any $f\in \Lambda_n$ the $u$-specialization $f[1-u]$ is
a polynomial in $u$ of degree at most $n$. We will need to consider
the effect of taking top degree coefficients in $u$ after
$u$-specialization. Define the  top degree of $f\in
\Lambda_n$ as
$$
[f]:=\left. u^nf\left[1-u^{-1}\right]\right|_{u=0}.
$$

It is a crucial fact for what follows that $u$-specialization and
taking its top degree coefficient commute with the $\Log$ map. More
precisely we have the following.
\begin{proposition}
\label{u-specializ-Log}
  Let $\Omega(\x;T)=\sum_{n\geq 0} A_n(\x)\ T^n\in \Lambda[[T]]$ be a
  power series with $A_n(\x)\in \Lambda_n$ and
  let $V_n(\x)\in \Lambda$ be defined by
$$
\sum_{n\geq 1} V_n(\x) \ T^n:=\Log\ \Omega(\x;T).
$$
Then we have

(i)
$$
\sum_{n\geq 1} V_n[1-u] \ T^n
= \Log \sum_{n\geq 0} A_n[1-u] \ T^n.
$$
and

(ii)
$$
\sum_{n\geq 1} [V_n] \ T^n
= \Log \sum_{n\geq 0} [A_n] \ T^n.
$$
\end{proposition}
\begin{proof}
Define $U_n(\x)\in \Lambda$ by 
$$
\sum_{n\geq 1} U_n(\x) \ \frac{T^n}n:=\log\ \Omega(\x;T).
$$
The relation between $U_n$ and $V_n$ is
\begin{equation}
\label{U-V}
V_n(\x):=\frac 1 n \sum_{d\mid n} \mu(d)\;
U_{n/d}(\x^d),
\end{equation}
where $\mu$ is the ordinary M\"obius function.

Since the $u$-specialization is a ring homomorphism we have
\begin{equation}
\label{u-specializ-log}
\sum_{n\geq 1} U_n[1-u] \ \frac{T^n}{n}= \log\
\sum_{n\geq 0} A_n[1-u] \ T^n.
\end{equation}
It is clear that $U_n(\x)$ is homogeneous of degree $n$ and hence we
may write it as $\sum_{|\lambda|=n}c_\lambda p_\lambda(\x)$ for some
coefficients $c_\lambda$. Therefore,
$$
V_n(\x)=\frac 1n \sum_{d \mid n} \mu(d) \sum_{|\lambda|=n/d} c_\lambda
p_\lambda(\x^d).
$$
Note that $p_\lambda(\x^d)=p_{d\lambda}(\x)$; so applying the
$u$-specialization to both sides we get
$$
V_n[1-u]=\frac 1n \sum_{d \mid n} \mu(d) \sum_{|\lambda|=n/d}
c_\lambda p_{d\lambda}[1-u].
$$
Similarly, since $p_{dr}[1-u]=1-u^{rd}=p_r[1-u^d]$, the inner sum on
the right hand side equals $U_{n/d}[1-u]$ evaluated at $u^d$ proving
(i).

To prove (ii) replace in~\eqref{u-specializ-log} $u$ by $u^{-1}$, $T$
by $uT$ and set $u=0$ to obtain
$$
\sum_{n\geq 1} [U_n] \ \frac{T^n}{n}= \log\
\sum_{n\geq 0} [A_n] \ T^n.
$$
 Now the claim follows from~\eqref{U-V}.
\end{proof}

\begin{proposition}
\label{HL-u-specializ}
For any partition $\lambda \in \calP$ we have 

(i)
$$
\tilde H_\lambda(q)[1-u]=(u)_l
$$
where $l:=l(\lambda)$ is the length of $\lambda$ and
$(u)_l:=\prod_{i=1}^l(1-q^{i-1}u)$. 

\medskip
(ii) $[\tilde H_\lambda]$ is zero unless $\lambda=(1^n)$
when it equals $(-1)^nq^{\binom n2}$.
\end{proposition}
\begin{proof}
  The specialization (i) follows from the corresponding result for
  Macdonald polynomials, see~\cite[Corollary 2.1]{garsia-haiman}.
 The second claim is an immediate consequence of (i).
\end{proof}

  We will need one last fact.
\begin{lemma} For the Schur function $s_\lambda$ we have that
  $s_\lambda[1-u]$ is zero unless $\lambda=(r,1^{n-r})$ with $1\leq r
  \leq n$ is a hook, in which case it equals $(-u)^{n-r}(1-u)$. In
  particular, for $f\in\Lambda_n$ we have
\begin{equation}
\label{leading-coeff-fmla}
[f]=(-1)^n\langle f, s_{(1^n)}\rangle.
\end{equation}
\end{lemma}
\begin{proof}
  The $u$-specialization of the Schur functions is given
  in~\cite[(2.15)]{garsia-haiman}. The
  identity~\eqref{leading-coeff-fmla} follows immediately.
\end{proof}

\subsection{DT-invariants}\label{DT2}

In this section we prove a somewhat more general case of
Proposition~\ref{formal} (ii). We work with a \emph{symmetric quiver}
(a quiver with as many arrows going from the vertex $i$ to $j$ as
arrows going from $j$ to $i$) instead of the double of a quiver (see Remark~\ref{parity}).
The only difference is that the double of a
quiver has an even number of loops at every vertex whereas a symmetric
quiver may not. We deal with this by attaching an arbitrary number of
legs to each vertex instead of just one. In general, the parity of the
number of legs required at a vertex $i$ is the opposite of that of the
number of loops at $i$.

Concretely, attach $k_i\geq 1$ infinite legs to each vertex $i\in I$
of $\Gamma$. The orientation of the arrows ultimately does not matter
but say all the arrows on the new legs point towards the vertex.
Consider the following generalization of~\eqref{maingen}
\begin{equation}
\label{gen-fctn}
\bH(\z;q):=(q-1)\Log\,\left(\sum_{\pihat\in \calP^r}
 \calH_\pihat(q)\tilde H_\pihat(\z;q)\right),
\end{equation}
where to simplify we let
$$
\tilde H_\pihat(\z;q):=\prod_{i=1}^r\prod_{j=1}^{k_i} 
\tilde H_{\pi^i}(\x^{i,j};q)
$$
and $\x^{i,j}=(x^{i,j}_1,x^{i,j}_2,\ldots)$ for $i=1,\ldots,r$ and
$j=1,\ldots, k_i$ are independent sets of infinitely many variables.

Given a multi-partition $\muhat=(\mu^{i,j})$ where $i=1,\ldots,r$ and
$j=1,\ldots,k_i$ define
$$
s_\muhat(\z):=\prod_{i=1}^r \prod_{j=1}^{k_i}s_{\mu^{i,j}}(\x^{i,j})
$$
and 
\beq
\Hs_\muhat(q):=\langle \bH(\z;q), s_\muhat(\z)\rangle.
\label{hua-u}
\eeq
Note that $\Hs_\muhat(q)$ is zero unless $|\mu^{i,1}|=|\mu^{i,2}|=\cdots
=|\mu^{i,k_i}|$ for each $i=1,\cdots,r$.

For $\v\in\N^r\backslash\{0\}$, denote by $1^\v$ the multi-partition
$(\mu^{i,j})$ where for every $j=1,\ldots,k_i$ either $\mu^{i,j}=(1^{v_i})$
if $v_i>0$ or $\mu^{i,j}=0$ otherwise.

\begin{proposition} We have
\beq
(q-1)\,\Log\,\left(\sum_{\v\in\N^r}\frac{
q^{-\frac12(\gamma(\v)+\delta(\v))} 
}{(q^{-1})_\v}\ (-1)^{\delta(\v)}T^\v \right)
=\sum_{\v\in\N^r\backslash\{0\}}
\Hs_{1^\v}(q)(-1)^{\delta(\v)}T^\v,
\label{hua-u2}
\eeq 
where 
$$
\gamma(\v):=\sum_{i=1}^r(2-k_i)v_i^2-
2\sum_{i\rightarrow j\in\Omega}v_iv_j,\qquad \delta(\v):=\sum_{i=1}^rk_iv_i,$$
and
$\qquad (q)_\v:=(q)_{v_1}\cdots (q)_{v_r}$ with $(q)_s:=(1-q)\cdots(1-q^s)$. 
\end{proposition}

\begin{proof}
By~\eqref{hua-u} we have
$$
 \bH(\z;q)=\sum_\muhat \Hs_\muhat(q) s_\muhat(\z).
$$
Apply Proposition~\ref{u-specializ-Log}~(ii) to all the variables
$\x^{i,j}$ in~\eqref{gen-fctn} to get
$$
\sum_\muhat \Hs_\muhat(q) [s_\muhat(\z)]T^{|\muhat|}=
(q-1)\,\Log\,\left(\sum_{\pihat\in \calP^r} 
 \calH_\pihat(q)[\tilde H_\pihat(\z;q)]T^{|\pihat|}\right),
$$
where $\muhat$ runs through the non-zero multi-partitions $(\mu^{i,j})$
with $|\mu^{i,j}|=v_i$ for some $v_i\in \N$ independent of $j$,
$T^{|\muhat|}:=\prod_iT_i^{v_i}$ and
$T^{|\pihat|}=\prod_iT_i^{|\pi_i|}$.  A calculation using
Proposition~\ref{HL-u-specializ}~(ii) shows that the right hand side
equals
$$
(q-1)\,\Log\,\left(\sum_{\v\in\N^r}\frac{
q^{-\frac12(\gamma(\v)+\delta(\v))} 
}{(q^{-1})_\v}\ (-1)^{\delta(\v)}T^\v \right)
$$
Finally,~\eqref{leading-coeff-fmla} shows that the left hand side
equals 
$$
\sum_{\v\in\N^r\backslash\{0\}}
\Hs_{1^\v}(q)(-1)^{\delta(\v)}T^\v
$$
and our claim is proved.
\end{proof}

Now let $\Gamma'=(I,\Omega')$ be any symmetric quiver with $r$
vertices. The Donaldson--Thomas invariants for a symmetric quiver
$\Gamma'$, as defined by Kontsevich and Soibelman, are given as
follows in an equivalent formulation.

 Let $c_{\v,k}$ be the coefficients in the generating
function identity
\begin{equation}
\label{DT-inv}
\Log \sum_\v\frac{
(-q^{\frac12})^{\gamma'(\v)}
}{(q)_\v}\ T^\v= (1-q)^{-1}\sum_\v \sum_k (-1)^kc_{\v,k}\ q^{k/2}\ T^\v,
\end{equation}
where
$$
\gamma'(\v):=\sum_{i=1}^rv_i^2-\sum_{i\rightarrow j\in\Omega'}v_iv_j.
$$
Put
$\Omega_\v(q):=\sum_k c_{\v,k}\, q^{k/2}$ ; it is a Laurent polynomial in
$q^{\tfrac12}$ (See~\cite[p.15]{efimov}.) 

Efimov~\cite[Thm. 4.1]{efimov} proves that if $c_{\v,k}$ is non-zero
then $k\equiv \gamma'(\v)\bmod 2$. Since $\Gamma'$ is symmetric we
also have $\gamma'(\v)\equiv \delta'(\v)\bmod 2$, where $\delta'$ is a
fixed linear form
$$
\delta'(v):=\sum_{i=1}^rk_i'v_i, \qquad k_i'\in\Z_{>0}, \qquad
k_i'\equiv a_{i,i}'-1 \bmod 2
$$
with $a_{i,i}'$ the number loops at the vertex $i$ of $\Gamma'$.
Hence we may write~\eqref{DT-inv} as
\begin{equation}
(1-q)\, \Log \sum_\v\frac{
q^{\frac12\gamma'(\v)}}
{(q)_\v}\ (-1)^{\delta'(\v)}T^\v= \sum_\v \Omega_\v(q)\
(-1)^{\delta'(\v)}T^\v. 
\end{equation}
Changing $q\mapsto q^{-1}$ and
then $T_i\mapsto q^{-k_i'/2}T_i$, we find that  
 \begin{equation}
\label{omega-fmla}
(q-1)\, \Log \sum_\v\frac{
q^{-\frac12 (\gamma'(\v)+\delta'(\v))}}
{(q^{-1})_\v}\ (-1)^{\delta'(\v)}T^\v= \sum_\v q^{1-\frac 12
  \delta'(\v)}\Omega_\v(q^{-1})\ (-1)^{\delta'(\v)}T^\v
\end{equation}

We extend the definition of $\DT_\v$ given in~\ref{DT-inv-defn} to
$\Gamma'$ by setting
 \begin{equation}
\label{DT-inv-defn-1}
(q-1)\, \Log \sum_\v\frac{
q^{-\frac12 (\gamma'(\v)+\delta'(\v))}}
{(q^{-1})_\v}\ (-1)^{\delta'(\v)}T^\v= 
\sum_{\v\in \N^r\backslash\{0\}}
\DT_\v(q)\,(-1)^{\delta'(\v)} T^\v.
\end{equation}
Up to powers of $q$ the definition of $\DT_\v(q)$ is independent of
the choice of linear form $\delta'$ and we do not include it in the
notation. Note that then
$$
\DT_\v(q)=q^{1-\frac 12
  \delta'(\v)}\Omega_\v(q^{-1}).
$$

We would like to match (\ref{DT-inv-defn-1}) with (\ref{hua-u2}) by
making appropriate choices for $\Gamma$ and $k_i$. Denote by $a_{i,j}$
(resp. $a_{i,j}'$) the number of arrows of $\Gamma$ (resp. $\Gamma'$)
going from $i$ to $j$. To match $\gamma$ with $\gamma'$ requires that
\beq a'_{i,j}+a'_{j,i}=2(a_{i,j}+a_{j,i}), \quad i\neq j, \qquad
\qquad k_i-2+2a_{i,i}=-1+a'_{i,i}.
\label{cond}\eeq
This we can always do (typically in more than one way) because
$\Gamma'$ is symmetric. We have then

\begin{proposition} With the above notation let $\Gamma$ be a quiver
  and $k_i$ be integers satisfying (\ref{cond}). Then
\begin{equation}
\label{DT-H-general}
\DT_\v(q)=q^{\delta(v)-\delta'(v)}\Hs_{1^\v}(q)
\end{equation}
for all $\v\in\N^r\backslash\{0\}$. 
\label{propDT}
\end{proposition}
A special case of Proposition~\ref{propDT} is when $\Gamma'=\overline
\Gamma$ for some quiver $\Gamma$. In this case we may take
$k_i=k_i'=1$ for all $i$ and~\eqref{DT-H-general} is claim (ii) of
Proposition~\ref{formal}.

\section{Examples}
{\it Example 1.} \ Let $\Gamma'=S_m$ be the quiver with one
node and $m$ loops. We can take $\Gamma$ to be the quiver with one
node and no arrows and take $k=k'=m+1$. Then
$\gamma'(n)=\gamma(n)=(1-m)n^2$, $\delta(n)=\delta'(n)=(m+1)n$ and
hence
$$
\sum_{n\geq1}\DT_n(q)\ (-1)^{(m-1)n}T^n=
(q-1)\ \Log \sum_{n\geq 0}
  \frac{q^{(m-1)\binom n 2-n}}{(q^{-1})_n}\; 
 (-1)^{(m-1)n}T^n.
$$
Up to a power of $q$ the invariants $\DT_n(q)$ are those considered by
Reineke~\cite{reineke}. Here is a list of the first few values. For
$m=0,1$ we have $\DT_1(q)=1$ and $\DT_n(q)=0$ for all $n> 1$.

\bigskip
$$
m=2
$$
$$
\begin{array}{l|l}
n & \DT_n(q)\\
\hline
1 & 1\\
2 & 1\\
3 & q\\
4 & q^3 + q\\
5 & q^6 + q^4 + q^3 + q^2 + q\\
6 & q^{10} + q^8 + q^7 + 2q^6 + q^5 + 3q^4 + q^3 + 2q^2 + q
\end{array}
$$
\bigskip
$$
m=3
$$
$$
\begin{array}{l|l}
n & \DT_n(q)\\
\hline
1 & 1\\
2 & q\\
3 & q^4 + q^2 + q\\
4 & q^9 + q^7 + q^6 + 2q^5 + q^4 + 2q^3 + q^2 + q\\
5 & q^{16} + q^{14} + q^{13} + 2q^{12} + 2q^{11} + 3q^{10} + 3q^9 + 4q^8 +
4q^7 + 5q^6 + 4q^5 + 4q^4 + 3q^3 + 2q^2 + q 
\end{array}
$$

\bigskip
\noindent {\it Example 2.} \ Let $\Gamma$ be the $A_2$ quiver, i.e.,
two nodes connected by an arrow, and let $k_1=k_2=1$. Then $\Gamma'$
is its double; the quiver with two nodes and an arrow between them in
each direction. In this case, $\tilde \Gamma$ is actually a finite
quiver (of type $A$). It follows that $\DT_n(q)$ is zero unless $\v$ is
a root of $A_2$; i.e. $\v=(0,1),(1,0)$, or $(1,1)$ and $\DT_\v(q)=1$
(see Corollary~\ref{pols-corollary}).

Using that
\begin{equation}
\label{q-fac}
(q)_n=(-1)^nq^{\binom n2 +n}(q^{-1})_n
\end{equation}
we can write~\eqref{hua-u2} as
$$
(q-1)\;\Log\; \sum_{n_1,n_2} \frac{q^{n_1n_2}}{(q)_{n_1}(q)_{n_2}} \;
  T_1^{n_1}T_2^{n_2}= -T_1-T_2+T_1T_2,
$$
which can be proved directly in an elementary way using the
$q$-binomial theorem (see~\cite[Prop. 1]{kashaev}). This identity is
related to the quantum pentagonal identity.

This case is the only case, other than the trivial quiver with one
vertex and no arrows, where $\tilde \Gamma$ is a finite quiver, or,
equivalently, where the right hand side of~\eqref{hua-u2} has only
finitely many terms.

\bigskip
\noindent {\it Example 3.} \ Let $\Gamma$ be the quiver with $r$ nodes
and $(m-1)\min(i,j)$ arrows pointing from vertex $i$ to $j$ and let
$k_i=2$ for all $i$. Then up to a power of $q$ the $\DT_\v(q)$
invariants equal the truncated form $A_\lambda(q)$ of the \Apol for
the $S_m$ quiver considered in~\cite{RV}. In particular, this shows
that these truncations $A_\lambda(q)$ are indeed, as conjectured,
polynomials in $q$ and have non-negative integer coefficients.

   \end{document}